\documentclass[11pt]{article}

\usepackage{amsmath,amsfonts,amsthm,amscd,amssymb,graphicx}
\usepackage{mathabx}

\usepackage[utf8]{inputenc}

\usepackage{enumitem}

\usepackage{subfigure}
\usepackage{xcolor}

\numberwithin{equation}{section}

\usepackage{hyperref}

\usepackage{cases}

\newtheorem{theorem}{Theorem}[section]

\newtheorem{lemma}[theorem]{Lemma}

\newtheorem{proposition}[theorem]{Proposition}

\newtheorem{corollary}[theorem]{Corollary}

\newcommand{\RR}{\mathbb{R}}
\newcommand{\cO}{\mathcal{O}}
\newcommand{\cS}{\mathcal{S}}

\newcommand{\cT}{\mathcal{T}}

\newcommand{\cR}{\mathcal{R}}

\def\threeplus{\left(\frac{1}{3}-\delta_1\right)^{-1}}
\def\sixminus{\left(\frac{1}{6}+\delta_1\right)^{-1}}

\begin{document}

\title{Large Time Behavior of the Klein-Gordon-Schr\"{o}dinger system} 

\author{Chanjin You\footnotemark[1]
}

\maketitle

\footnotetext[1]{Penn State University, Department of Mathematics, State College, PA 16802. Email: cby5175@psu.edu.}


\begin{abstract}

We establish the global existence and scattering for small and localized solutions of the Klein-Gordon-Schr\"{o}dinger system in three dimensions. The system consists of coupled semilinear Schr\"{o}dinger and Klein-Gordon equations with quadratic nonlinearities. This model is motivated by the study of plasma oscillations arising from the Hartree equation near a translation-invariant equilibrium with the Coulomb potential. Our proof relies on the space-time resonance method. The main difficulty comes from the two dimensional space-time resonant set and the absence of null form structure.

\end{abstract}

\tableofcontents


\section{Introduction}
We are interested in the large time behavior of the Klein-Gordon-Schr\"{o}dinger system given by
\begin{equation}\label{eq:KGS}
	\begin{cases}
		\begin{aligned}
			& i\partial_t u + \Delta u =  vu, \\
			& (\Box + 1) v =  \pm |u|^2, \\
			& u_{\vert_{t=0}} = u_0, \quad v_{\vert_{t=0}} = v_0, \quad \partial_t v_{\vert_{t=0}} = v_1,
		\end{aligned}
	\end{cases}
\end{equation}
where $u : \mathbb{R}^{+} \times \mathbb{R}^3 \to \mathbb{C}$ and $v : \mathbb{R}^{+} \times \mathbb{R}^3 \to \mathbb{R}$, with $\Box = \partial_t^2 - \Delta_x$.

This system models scalar nucleons interacting with neutral scalar mesons. The nucleons and mesons are described by $u$ and $v$, respectively.
It can also be regarded as a Davey-Stewartson type equation \cite{Zakharov1991}, which is in general of the form
\begin{equation}
	\begin{cases}
		\begin{aligned}
			& i\partial_t u + L_1 u = vu, \\
			& L_2 v  = L_3|u|^2,
		\end{aligned}
	\end{cases}
\end{equation}
where $L_1$, $L_2$, and $L_3$ are differential operators with constant coefficients.

The system \eqref{eq:KGS} also naturally arises from the study of mean-field dynamics for many-body quantum systems. Indeed, consider the Hartree equation
\begin{equation}\label{eq:Hartree}
	\begin{cases}
		\begin{aligned}
			& i\partial_t \gamma = [-\Delta + w \star_x \rho_{\gamma}, \gamma] \\
			& \gamma \rvert_{t=0} = \gamma_0
		\end{aligned}
	\end{cases}
\end{equation}
where $\gamma(t)$ is the one-particle density operator on $L^2(\mathbb{R}^3)$, and $\rho_{\gamma}(t)$ is the density associated with $\gamma(t)$ defined by $\rho_{\gamma}(t,x) = \gamma(t,x,x)$. 
The system \eqref{eq:Hartree} has translation-invariant equilibria $f= f(-\Delta)$. When the interaction potential $w$ is smooth and rapidly decaying, a certain class of equilibria is Penrose stable, and their long-time behavior has been studied in \cite{Lewin2014, Chen2018}. The smoothness and decay assumption on the interaction potential was relaxed in \cite{Collot2020, Collot2022, Hadama2023a, You2024, Borie2025}, showing that $\widehat{w} \in L^{\infty}$ corresponds to a scattering subcritical case for Penrose stable steady states. The initial perturbation was assumed to be sufficiently regular and rapidly decaying. These are extended to low regularity initial perturbation in \cite{Hadama2023a, Borie2025}. We also remark that \cite{Lewin2020, Smith2024} considered the semiclassical limit $\hbar \to 0$ of the Hartree equation near a translation-invariant equilibrium, recovering the Vlasov equation near a homogeneous equilibrium, whose asymptotic stability is related to Landau damping in plasma physics.

For the Coulomb potential $w(x) = |x|^{-1}$, $x \in \mathbb{R}^3$, it was proven in \cite{Nguyen2023a} that the Penrose type stability criterion fails for every translation-invariant equilibrium, and the leading behavior of the density associated with the perturbation is dominated by the Klein-Gordon type dispersion relation, which arises from plasmons. It follows that the density decays at a rate of $t^{-3/2}$, which is slow compared to the $t^{-3}$ decay in the free Hartree case. This is due to the long-range nature of the interaction potential. Consequently, the analysis of the full nonlinear problem should incorporate the \textit{particle-wave interaction}, as opposed to the scattering subcritical case where the phase mixing addresses particle-particle interactions \cite{You2024, Smith2024}. In view of this, it is natural to consider the coupled system
\begin{equation}\label{eq:Hartree-KG}
	\begin{cases}
		\begin{aligned}
			& i\partial_t \gamma = [-\Delta + V_{\gamma}, \gamma], \\
			& (\Box +1) V_{\gamma}= \rho_{\gamma}, \\
			&\gamma \rvert_{t=0} = \gamma_0,
		\end{aligned}
	\end{cases}
\end{equation}
which models a system of quantum particles interacting with the Klein-Gordon field. For other models describing particle-wave interactions, refer to \cite{Hani2013, Beck2015, Nguyen2024a, Nguyen2024b}.

Observe that  when $\gamma(t)$ is a pure state, that is, $\gamma(t) = \lvert u(t) \rangle \langle u(t) \rvert$ for some $u \in L^2(\mathbb{R}^3)$, the system \eqref{eq:Hartree-KG} reduces to \eqref{eq:KGS}. In \cite{Baillon1978}, the global wellposedness of strong solutions in $\mathbb{R}^3$ was established. In \cite{Bachelot1984}, the global existence of weak solutions in $L^2(\mathbb{R}^3) \cap L^4(\mathbb{R}^3)$ for large initial data was proved. Low regularity well-posedness results were established in \cite{Pecher2012a, Pecher2012}.
The existence of (modified) wave operators for the system \eqref{eq:KGS} in two dimensions has been studied in \cite{Ozawa1994, Shimomura2003, Shimomura2004, Shimomura2005}. Small data scattering in Lorentz spaces is established in \cite{Banquet2015}.

In this paper, we establish the global existence and scattering for small and localized solutions of \eqref{eq:KGS} in Sobolev spaces. This system lacks a scaling invariant transformation due to the Klein-Gordon part. Also, it has mixed linear parts. Even the speeds of propagation are different for the Schr\"{o}dinger part and the Klein-Gordon part. The former has an infinite speed of propagation, while the latter has a finite one. Consequently, there are not enough vector fields commuting with both linear Schr\"{o}dinger and Klein-Gordon equations, which makes it difficult to rely solely on the vector field method. 

\subsection{Main result}
Now we state our main result. For small initial data $(u_0, v_0, v_1)$, the system \eqref{eq:KGS} has a global-in-time solution $(u(t), v(t))$ in $H^{N} \times H^{N}$ and moreover, it scatters.
\begin{theorem}\label{thm:main}
	Let $N \ge 4000$ and
	\begin{equation}\label{assumption:initialdata}
		\epsilon_0:= \| u_0 \|_{H^{N}} + \| x u_0 \|_{H^1} + \| v_0 \|_{H^{N}} + \| v_1 \|_{H^{N-1}} + \| x v_0 \|_{H^4} + \left\| x v_1 \right\|_{H^3} < \infty.
	\end{equation}
	If $\epsilon_0$ is sufficiently small, then there exists a unique global solution $(u(t), v(t))$ to \eqref{eq:KGS} in $H^{N} \times H^{N}$ satisfying
	\[
	\| u(t) \|_{W^{1,\threeplus}}\lesssim \epsilon_0 \langle t \rangle^{-1/2-3\delta_1}, \qquad  
	\| v(t)\|_{L^{\sixminus}} \lesssim \epsilon_0 \langle t \rangle^{-1 + 3\delta_1}
	\]
	for some small $\delta_1 > 0$ and for all $t\ge 0$. Moreover, $(u(t), v(t))$ scatters to a free solution as $t \to\infty$. Precisely, there exist solutions $u_{\infty}(t)$, $v_{\infty}(t)$ to free equations
	\[
	i\partial_t u_{\infty} + \Delta u_{\infty} =0, \qquad (\Box +1)v_{\infty} =0 
	\]
	satisfying	
	\[
	\| u(t) - u_{\infty}(t)\|_{H^{N}} + \| v(t) - v_{\infty}(t) \|_{H^{N}} \lesssim \epsilon_0^2 \langle t \rangle^{-3\delta_1/2}
	\]
	for $t \ge 0$.
\end{theorem}

Our main result also applies to the generalized Klein-Gordon-Schr\"{o}dinger system \eqref{eq:generalizedKGS}, where the Klein-Gordon symbol $\langle k \rangle$ is replaced with the Klein-Gordon type symbol $\nu(|k|)$. Precisely, we obtain the following:
\begin{theorem}\label{thm:generalizedKGS}
	Under the assumptions in Theorem \ref{thm:main}, the same conclusion holds for the system
	\begin{equation}\label{eq:generalizedKGS}
		\begin{cases}
			\begin{aligned}
				& i\partial_t u + \Delta u =  vu, \\
				& (\partial_t^2 + \nu^2(|\nabla|) )v =  \pm |u|^2, \\
				& u_{\vert_{t=0}} = u_0, \quad v_{\vert_{t=0}} = v_0, \quad \partial_t v_{\vert_{t=0}} = v_1,
			\end{aligned}
		\end{cases}
	\end{equation}
	where $\nu (|k|)$ is of class $C^{N_0}$ for some $N_0 \ge 4$ and satisfies
	\begin{align*}
		c_0 \langle k \rangle \le \nu (|k|) &\le C_0 \langle k \rangle, \\[2pt]
		c_0 \frac{|k|}{\langle k \rangle} \le \nu'(|k|) &\le C_0 \frac{|k|}{\langle k \rangle}, \\[2pt]
		c_0 \langle k \rangle^{-3} \le \nu''(|k|) &\le 2, \\[2pt]
		|\nu^{(n)}(|k|)| &\le C_0, \qquad 3 \le n \le N_0
	\end{align*}
	uniformly in $k \in \mathbb{R}^3$, for some positive constants $c_0$ and $C_0$.
\end{theorem}
The assumptions on $\nu(|k|)$ are standard for Klein-Gordon type symbols except for the condition $\nu''(|k|) \le 2$, which is imposed to ensure that the structure of resonant sets remains similar to that of the previous system.

	\subsection{Organization of the paper}
	In Section \ref{sec:prelim}, we begin with a brief introduction to the space-time resonance method. Then we collect notations and preliminary estimates, and specify the parameters and cutoffs that will be used repeatedly in this paper.
	In Section \ref{sec:nonlineariteration}, we describe the bootstrap argument and then reduce the problem to bilinear estimates by analyzing the contributions of initial data and small time terms. Then we show that high frequency terms can be controlled using bounds on high Sobolev norms.
	In Sections \ref{sec:F} and \ref{sec:G}, we establish the energy and decay estimates for $F_{\pm}$ and $G$, respectively, completing the proof of Theorem \ref{thm:main}. Finally, in Section \ref{sec:generalizedKGS}, we consider the generalized system \eqref{eq:generalizedKGS} and sketch the proof of Theorem \ref{thm:generalizedKGS}.

\section{Preliminaries}\label{sec:prelim}
\subsection{Space-time resonances}
Our analysis relies on the space-time resonance method, developed independently by \cite{Germain2009, Gustafson2009}. Consider the bilinear operator
\[
T_m (f_1,f_2) (x) = \mathcal{F}^{-1}_{\xi \to x} \int_{0}^{t} \int_{\RR^3} e^{is \Phi(\xi,\eta)} m(\xi, \eta) \widehat{f}_1(s,\xi-\eta) \widehat{f}_2(s,\eta)\, d\eta \, ds.
\]
A direct computation shows that we have
\[
e^{is\Phi} = \frac{1}{i\Phi} \partial_s e^{is \Phi}
\]
for $\Phi \neq 0$, and
\[
e^{is\Phi} = \frac{\partial_\eta \Phi}{is |\partial_\eta \Phi|^2}  \cdot \partial_\eta e^{is\Phi},
\]
for $\partial_{\eta} \Phi \neq 0$.
In view of this, we define the sets of time, space, and space-time resonant frequencies as follows.
\begin{equation*}
	\begin{aligned}
		\mathcal{T} &= \{ (\xi,\eta) \in \mathbb{R}^3 \times \mathbb{R}^3, \quad \Phi(\xi,\eta) =0 \}, \\
		\mathcal{S} &= \{ (\xi, \eta) \in \mathbb{R}^3 \times \mathbb{R}^3, \quad \partial_\eta \Phi(\xi, \eta) =0 \}, \\
		\mathcal{R} &= \mathcal{T} \cap \mathcal{S}.
	\end{aligned}
\end{equation*}
On $\mathcal{T}$ and $\mathcal{S}$, the phase $\Phi$ is stationary in $s$ and $\eta$, respectively. In view of the stationary phase argument, the main contribution comes from the set $\cR$.

\subsection{Setup}
In order to write the given equations in the first order in time, we introduce
\[
V_{+} = V= (\partial_t - i \langle \nabla \rangle) v, \qquad V_{-} = \overline{V} = (\partial_t + i\langle \nabla \rangle) v
\]
where $\langle \nabla \rangle = \sqrt{1- \Delta}$. Then we have
\[
v= \frac{i}{2\langle \nabla \rangle} (V- \overline{V}) = \sum_{\pm} a_{\pm}(\nabla) V_{\pm},
\]
where $a_{\pm}(\eta) = \pm \frac{i}{2} \langle \eta \rangle^{-1}$, and $\langle \eta \rangle= \sqrt{1+|\eta|^2}$.
Now \eqref{eq:KGS} can be written as
\begin{equation}
	\begin{cases}
		\begin{aligned}
			& (\partial_t  -i \Delta )u =  -iu\sum_{\pm} a_{\pm}(\nabla) V_{\pm} , \\
			& (\partial_t + i\langle \nabla \rangle)V =|u|^2, \\
			& u_{\vert_{t=0}}= u_0, \quad V_{\vert_{t=0}} = v_1 - i \langle \nabla \rangle v_0.
		\end{aligned}
	\end{cases}
\end{equation}
Define the associated profiles by
\[
f(t) = e^{-it \Delta} u(t), \qquad g_{\pm}(t) = e^{\pm it\langle \nabla \rangle} V_{\pm}(t).
\]
Note that
\[
\widehat{f}(t,\xi) = e^{it|\xi|^2} \widehat{u}(t,\xi), \qquad \widehat{g}_{\pm}(t,\xi) = e^{\pm it \langle \xi \rangle} \widehat{V}_{\pm}(t,\xi).
\]
Applying Duhamel's formula, we get
\begin{align}
	& \widehat{f}(t,\xi) = \widehat{u}_{0} (\xi) -i  \sum_{\pm}  \int_{0}^{t} \int e^{is\Phi_{f,\pm}(\xi, \eta)} a_{\pm}(\eta)  \widehat{f}(s, \xi- \eta)  \widehat{g}_{\pm}(s, \eta) \, d\eta \, ds ,  \\
	& \widehat{g}(t,\xi) = \widehat{V}_0(\xi) + \int_{0}^{t} \int e^{is \Phi_{g} (\xi, \eta)} \widehat{f}(s,\xi-\eta) \overline{\widehat{f}(s,-\eta)} \, d\eta \, ds \label{eq:duhamel-KG}
\end{align}
where $g= g_{+}$, $\widehat{V}_0(\xi) = \widehat{v}_1 (\xi) - i \langle \xi \rangle \widehat{v}_0 (\xi)$, and
\begin{equation*}
	\begin{split}
		\Phi_{f+} (\xi, \eta) &= |\xi|^2 - \langle \eta \rangle - |\xi-\eta|^2 , \\
		\Phi_{f-} (\xi, \eta) &= |\xi|^2 + \langle \eta \rangle - |\xi-\eta|^2 , \\
		\Phi_{g} (\xi, \eta) &= \langle \xi \rangle - |\xi - \eta|^2 + |\eta|^2. 
	\end{split}
\end{equation*}
By direct computations, the resonant sets for $\Phi_{f+}$ are given by
\begin{equation}\label{eq:resonance-f+}
	\begin{split}
		\mathcal{T}_{f+} 
		&= \left\{ 2 \xi \cdot \eta =  \langle \eta \rangle + |\eta|^2\right\}, \\[2pt]
		\mathcal{S}_{f+} 
		&= \left\{ \xi = \eta \left( 1 + \frac{1}{2\langle \eta \rangle} \right) \right\}, \\[2pt]
		\mathcal{R}_{f+} 
		&= \left\{ \xi = \lambda \eta, \ |\eta| = R\right\},
	\end{split}
\end{equation}
where $R$ is the unique positive number satisfying $R^2 \sqrt{1+R^2} =1$, and $\lambda = 1+\frac{1}{2\langle R \rangle}$. 
For $\Phi_{f-}$,  we have 
\begin{equation}\label{eq:resonance-f-}
	\begin{split}
		\mathcal{T}_{f-} 
		&= \left\{ 2 \xi \cdot \eta = -\langle \eta \rangle + |\eta|^2 \right\},\\[2pt]
		\mathcal{S}_{f-} 
		&= \left\{ \xi = \eta \left( 1- \frac{1}{2\langle \eta \rangle} \right) \right\}, \\[2pt]
		\mathcal{R}_{f-} 
		&= \emptyset,
	\end{split}
\end{equation}
and for $\Phi_{g}$, we get
\begin{equation}\label{eq:resonance-g}
	\begin{split}
		\mathcal{T}_{g} 
		&= \left\{ 2\xi \cdot \eta =-\langle \xi \rangle +|\xi|^2  \right\}, \\[2pt]
		\mathcal{S}_{g} 
		&= \{ \xi = 0 \}, \\[2pt]
		\mathcal{R}_{g} 
		&= \emptyset.
	\end{split}
\end{equation}
Our resonant sets are generic in the sense that $\mathcal{R}_{f+}$ is a two dimensional manifold of the form $\cR = \{ \xi= \lambda \eta, \ |\eta| = R\}$, which is expected in the case when the dispersion relations $P(\xi)$ depend only on $|\xi|$.
Moreover, the equation does not have a meaningful null form structure, unlike the models studied in \cite{Pusateri2013, Hani2013, Beck2015}.

It should be also noted that $\operatorname{dist}(\mathcal{T}_g, \mathcal{S}_g) =0$, although $\mathcal{R}_{g} = \emptyset$. It turns out that these two sets are asymptotically close to each other when $|\xi| \ll 1 $ and $|\eta| \gg 1$. In order to resolve this issue, we first cutoff the high frequencies $|(\xi,\eta)| \gtrsim  s^{\delta_3}$ for $\delta_3 \le \frac{2}{N-4}$, which is harmless due to the control on high Sobolev norms, see Section \ref{sec:highfreq} for details.

	Finally, we note that the resonances are separated. Define
	\[
		\cR := \bigcup_{\iota} \cR_{\iota}
	\]
	as the union of all space-time resonant sets, where $\iota \in \{ f+, f-, g\}$.
	Let $\pi_{\xi} (\xi', \eta') = \xi'$, $\pi_{\eta}(\xi', \eta') =\eta'$ and $\pi_{\xi-\eta} (\xi', \eta') = \xi' - \eta'$. Define the set $\cO$ of outcome frequencies and the set $\mathcal{G}$ of germ (or source) frequencies as
	\begin{align*}
	\cO = \pi_{\xi} (\cR), \qquad
	\mathcal{G} = \pi_{\eta}(\cR) \cup \pi_{\xi-\eta} (\cR).
	\end{align*}
	In \cite{Germain2011, Germain2014}, the resonances are said to be separated if $\mathcal{G} \cap \mathcal{O} = \emptyset$. In this case, resonances are not driven by the interaction of resonances.
	For the system \eqref{eq:KGS}, we have
	\begin{align*}
		\cO = B_{\lambda R} (0), \qquad
		\mathcal{G} = B_{R}(0) \cup B_{(\lambda-1)R } (0).
	\end{align*}
	Since $R \neq 0$ and $\lambda >1$, the separation of resonance condition holds. Note that $R\in (0.868, 0.869)$, and $r_0:= R( 1 + \frac{1}{2\langle R \rangle})\in(1.196, 1.197)$. Then we can choose $\delta_0 >0$ such that
	\[
		B_{10\delta_0}(\cO) \cap \mathcal{G} = \emptyset.
	\]
	Define a smooth cutoff function $\chi_{\mathcal{O}}$ valued in $[0,1]$,  $\chi_{\mathcal{O}} =1 $ on $B_{\delta_0 /2} (\mathcal{O})$ and $\chi_{\mathcal{O}} = 0$ outside of $B_{\delta_{0}} (\mathcal{O})$. Define $\widetilde{\chi}_{\mathcal{O}} = 1 - \chi_{\mathcal{O}}$. Let $Z_{\cO} =  \chi_{\cO}(\nabla)$ and $\widetilde{Z}_{\cO} = \widetilde{\chi}_{\cO}(\nabla)$ be the corresponding operators.

\subsection{Littlewood-Paley}\label{sec:LP}

Let $\varphi$ be a smooth cutoff function taking values in $[0,1]$ such that $\varphi(\xi) =1$ for $|\xi| \le 1$ and $\varphi(\xi) = 0$ for $|\xi| \ge 2$. Define $\psi(\xi)= \varphi(\xi) - \varphi(2\xi)$, which is supported in $\frac{1}{2} \le |\xi| \le 2$. Then define $\varphi_{k}(\xi) = \varphi(2^{-k} \xi)$ and $\psi_{k}(\xi) = \psi(2^{-k} \xi)$ for $k \in \mathbb{Z}$. We have
\[
\sum_{k \in \mathbb{Z}} \psi_{k}(\xi) = 1, \qquad \xi \neq 0,
\]
and
\[
\varphi(\xi) + \sum_{k=1}^{\infty} \psi_{k}(\xi) =1
\]
for all $\xi$. We define the Littlewood-Paley projections $P_{\le k}$ and $P_{k}$ via
\[
\widehat{P_{\le k} f } = \varphi_{k}(\xi) \widehat{f}(\xi) , \qquad \widehat{P_{k} f} = \psi_{k}(\xi) \widehat{f}(\xi).
\]
We recall the classical Bernstein inequalities
\begin{equation}\label{ineq:Bernstein}
\| P_{k} f \|_{L^p} \lesssim 2^{3k\left( \frac{1}{q} - \frac{1}{p} \right)} \| P_{k} f \|_{L^{q}}, \qquad
\| P_{\le k} f \|_{L^p} \lesssim 2^{3k \left( \frac{1}{q} - \frac{1}{p} \right)} \| P_{\le k} f \|_{L^q},
\end{equation}
for $1 \le q \le p \le \infty$ and $k \in \mathbb{Z}$. 

\subsection{Linear estimates}
We begin with classical dispersive estimates for linear Schr\"{o}dinger \cite{Tao2006} and Klein-Gordon equations \cite{DAncona2008}.
\begin{lemma}\label{lem:disp-Sch}
	Let $p, q\in [2,\infty]$ and $p', q'\in [1,2]$ satisfy $1/p + 1/p' = 1$, $1/q + 1/q'=1$, and $2/p + 3/q = 3/2$. Then there hold
	$$\|e^{it\Delta} f \|_{L^p} \lesssim t^{-3 \left( \frac{1}{2} - \frac{1}{p} \right)} \| f \|_{L^{p'}}$$
	and
	$$\left\| \int_{0}^{t} e^{is\Delta} F(s) \, ds \right\|_{L^2_{x}} \lesssim \| F \|_{L^{p'}_t L^{ q'}_x}.$$
\end{lemma}

\begin{lemma}\label{lem:disp-KG} 
	Let $\epsilon >0$. Let $p, q\in [2,\infty]$ and $p', q'\in [1,2]$ satisfy $1/p + 1/p' = 1$, $1/q + 1/q'=1$, and $2/p + 3/q = 3/2$. Then there hold
	$$\|e^{it\langle \nabla \rangle} f \|_{L^p} \lesssim t^{-3 \left( \frac{1}{2} - \frac{1}{p} \right)} \| f \|_{W^{ 5\left(\frac{1}{2} - \frac{1}{p} \right) + \epsilon, p'}}$$
	and
	$$\left\| \int_{0}^{t} e^{is\langle \nabla \rangle} F(s) \, ds \right\|_{L^2_{x}} \lesssim \| F \|_{L^{p'}_t W^{- \frac{1}{q} + \frac{1}{p} + \frac{1}{2} + \epsilon, q'}_x}.$$
\end{lemma}

Next, we often use the following estimates on Fourier multipliers.
\begin{lemma}
	For $s \in [0, 3/2)$, there holds
	\[
	\| m(\nabla) f \|_{L^2} \lesssim \| m \|_{L^{3/s}} \| x^s f \|_{L^2}.
	\]
	In particular, for $\rho \in (0,1]$ and  $\varphi$ defined in Section \ref{sec:LP}, we have
	\begin{equation}\label{eq:symbol-resonantcase-linear}
		\left\| \varphi\left( \frac{|\nabla| - R}{\rho} \right) f \right\|_{L^2} \lesssim \rho^{s/3} \| x^s f \|_{L^2}.
	\end{equation}
\end{lemma}

\begin{proof}
	By Plancherel theorem and Sobolev embedding, we get
	\[
	\| m(\nabla) f \|_{L^2} \lesssim \| m\|_{3/s} \| \widehat{f} \|_{\left( \frac{1}{2} - \frac{s}{3} \right)^{-1}} \lesssim \| m \|_{3/s} \| |\nabla|^s \widehat{f} \|_{2} = \| m \|_{3/s} \| |x|^{s} f \|_{2}.
	\]
	Next, we observe that
	\[
		\left\| \varphi \left( \frac{|\xi|-R}{\rho} \right) \right\|_{3/s} \lesssim \rho^{s/3},
	\]
	which completes the proof of \eqref{eq:symbol-resonantcase-linear}.
\end{proof}

\subsection{Bilinear estimates}
Define
\[
T_m (f, g)(x) = \mathcal{F}^{-1} \int_{\RR^3}  m(\xi, \eta) \widehat{f}(\xi-\eta) \widehat{g}(\eta)\, d\eta .
\]
We recall the classical Coifman-Meyer theorem here.
\begin{lemma}\label{lem:Coifman-Meyer}
	For $p,q, r \in [1,\infty]$ satisfying $1/p + 1/q = 1/r$, there holds
	\[
	\| T_{m}(f,g) \|_{L^r} \lesssim \| m \|_{S^{\infty}} \| f \|_{L^p} \| g \|_{L^q}
	\]
	where
	\[
	\| m \|_{S^{\infty}}: = \| \widehat{m} \|_{L^1_{x,y}} \lesssim \| m \|_{H^{3/2+\epsilon}_{\xi,\eta}}
	\]
	for any $\epsilon >0$.
\end{lemma}

\begin{proof}
	The first inequality is standard, see \cite[Proposition 6.2]{Germain2011} for the proof. The second inequality follows from 
	\[
		\| \widecheck{m} \|_{L^1_{x,y}} \le \| \langle x \rangle^{-s} \langle y \rangle^{-s} \|_{L^2_{x,y}} \| \langle x \rangle^s \langle y \rangle^{s} \widecheck{m} \|_{L^2_{x,y}} \lesssim \| m\|_{H^{s}_{\xi,\eta}}
	\]
	which holds for $s > 3/2$.
\end{proof}

For certain type of symbols, their $S^{\infty}$ norms cannot be evaluated but similar bilinear estimates hold.
\begin{lemma}\label{lem:symbol-resonantcase-bilinear}
	Let $\chi$ be a smooth compactly supported funtion. If $m(\xi, \eta) =\chi(\xi - \lambda \eta)$, then
	\[
	\| T_{m}(f,g) \|_{L^r} \lesssim \| \widehat{\chi} \|_{L^1} \| f \|_{L^p} \|g \|_{L^q}
	\]
	In particular, for $\varphi$ defined in Section \ref{sec:LP}, there holds
	\[
	\| T_{\varphi (\frac{\xi-\lambda\eta}{\rho})}(f,g) \|_{L^r} \lesssim \| f \|_{L^p} \|g \|_{L^q}
	\]
	uniformly in $\rho \in(0,1]$.
\end{lemma}
\begin{proof}
	For the proof of the first inequality, see \cite[Proposition 6.4]{Germain2011}. The second inequality follows from	the identity $\| \widehat{\chi}_{\rho}\|_{L^1} = \| \widehat{\varphi}\|_{L^1}$ where $\chi_{\rho}(\xi):= \varphi(\xi/\rho)$.
\end{proof}
The cutoff functions used throughout this paper are summarized below for reference.
\begin{align*}
	m_{1\pm} (\xi,\eta)
	&= \widetilde{\chi}_{\cO}(\xi)\varphi \left( \frac{(\xi,\eta)}{Ms^{\delta_3}}\right)  \frac{\chi_{\cS_{f\pm}}}{i\Phi_{f\pm}} (\xi,\eta) a_{\pm}(\eta), \\
	m_{2\pm} (\xi,\eta)
	&= \widetilde{\chi}_{\cO}(\xi)  \varphi \left(\frac{(\xi,\eta)}{Ms^{\delta_3}}\right) \frac{\chi_{\cT_{f\pm}} \partial_{\eta} \Phi_{f\pm}}{i|\partial_{\eta} \Phi_{f\pm}|^2} (\xi,\eta) a_{\pm}(\eta), \\
	m_{3}^{s} (\xi,\eta)  &= \chi_{\cR}^{1}  \chi_{\cS_{f+}}^{s^{-\delta_2}} \frac{1}{i\Phi_{f+}} (\xi,\eta), \\
	m_{4}^s (\xi,\eta) &= \chi_{\cR}^{1} \chi_{\cT_{f+}}^{s^{-\delta_2}} \frac{\partial_{\eta} \Phi_{f+}}{i|\partial_{\eta} \Phi_{f+}|^2} (\xi,\eta), \\
	m_{5} (\xi,\eta) &= \frac{ \chi_{\cS_{g}}}{i\Phi_g} \partial_{\xi} \Phi_g (\xi,\eta) \chi_{\cO}(\xi-\eta) \chi_{\cO}(\eta), \\
	m_{6} (\xi,\eta) &=  \frac{\chi_{\cT_{g}} \partial_{\eta} \Phi_g}{i|\partial_{\eta} \Phi_g|^2}  \partial_{\xi}\Phi_g (\xi,\eta)  \chi_{\cO}(\xi-\eta) \chi_{\cO}(\eta), \\
	m_{7} (\xi,\eta) &= \varphi \left( \frac{(\xi,\eta)}{Ms^{\delta_3}}\right) \chi_{\cS_{g}}^{s^{-\delta_3}}  \frac{1}{i\Phi_g} (\xi,\eta), \\
	m_{8} (\xi,\eta) &= \varphi \left( \frac{(\xi,\eta)}{Ms^{\delta_3}}\right) \chi_{\cT_{g}}^{s^{-\delta_3}}  \frac{\partial_{\eta} \Phi_g}{i|\partial_{\eta} \Phi_g|^2} (\xi,\eta).
\end{align*}
Here $\delta_2$ and $\delta_3$ are small parameters to be determined later, see Section \ref{sec:parameters}. We have the following estimates for these symbols.

\begin{corollary}\label{coro:symbolests}
	Let $s \ge 1$. There exists $A >0$ such that
	\begin{align*}
		\| m_{1\pm} \|_{S^{\infty}} + \| m_{2\pm} \|_{S^{\infty}} +   \| m_{7} \|_{S^{\infty}} + \| m_{8} \|_{S^{\infty}} 
		&\lesssim s^{A\delta_3}, \\
		\|m_{3}^s\|_{S^{\infty}} + \|m_4^{s} \|_{S^{\infty}} 
		&\lesssim s^{A\delta_2}, \\
		\| m_{5} \|_{S^{\infty}} + \| m_{6} \|_{S^{\infty}} &\lesssim 1,
	\end{align*}
	uniformly in $s$.
\end{corollary}

\subsection{Choice of parameters}\label{sec:parameters}
We provide a summary of the choice of parameters for completeness. Let $A>0$ be chosen to satisfy the estimates in Lemma \ref{coro:symbolests}. It can be shown that $A=10$ is acceptable. Then we choose parameters $N$, $\delta_1$, $\delta_2$, and $\delta_3$ in such a way that
\[
1 \gg \delta_2 \gg \delta_1 \gg \delta_3 >0, \quad \textrm{and} \quad N \delta_3 \gg 1.
\]
More precisely, we impose
\[
\frac{1}{2(A+1/6)} > \delta_2 > 18\delta_1 > 6 (A+3)\delta_3, \quad \textrm{and} \quad \delta_3 \ge \frac{2}{N-4},
\]
see Section \ref{sec:highfreq}, \ref{sec:F}, and \ref{sec:G} for details. For convenience, we may take
\[
A=10, \quad N= 4000, \quad \delta_1 = 2.2 \times 10^{-3}, \quad \delta_2 = 4 \times 10^{-2},\quad \delta_3 = 5.05 \times 10^{-4}.
\]
We remark that no attempts were made to optimize the parameters.

\subsection{Construction of cutoff functions}\label{sec:cutoff}
We often decompose the frequency space to perform integration by parts.
To define the high frequency cutoff, we choose a large $M>0$ satisfying $\mathcal{R} \subseteq B_{M/2} (0)$.
Let $\varphi$ be a smooth cutoff function taking values in $[0,1]$ and satisfying $\varphi(x) = 1$ for $|x| \le 1$ and $\varphi(x) = 0$ for $|x| \ge 2$. 
Then we define smooth cutoff functions $\chi^{\rho}_{\mathcal{T}_{\iota}}$, $\chi^{\rho}_{\mathcal{S}_{\iota}}$ and $\chi^{\rho}_{\mathcal{R}_{\iota}}$ for $\iota \in \{ f+, f-, g\}$, following the argument in \cite{Germain2011}. It should be considered that $\partial_{\xi,\eta}^{\alpha} \Phi_{\iota}$ may not be bounded for the system \eqref{eq:KGS}.

\begin{lemma}\label{lem:cutoff-f-}
	There exist smooth cutoff functions $\chi_{\cT_{f-}}$ and $\chi_{\cS_{f-}}$, taking values in $[0,1]$, satisfying
	\begin{itemize}
		\item $\chi_{\cT_{f-}} + \chi_{\cS_{f-}} =1$.
		\item For all $\alpha$, there holds
		\begin{equation}\label{est:cutoff-f-}
		\left| \partial_{\xi,\eta}^{\alpha} \frac{\chi_{\cS_{f-}}}{\Phi_{f-}} \right| + \left| \partial_{\xi,\eta}^{\alpha} \frac{\chi_{\cT_{f-}} \partial_{\eta} \Phi_{f-} }{|\partial_{\eta} \Phi_{f-}|^2} \right|
		\lesssim \langle \xi,\eta\rangle ^{|\alpha|}.
		\end{equation}
	\end{itemize}
\end{lemma}
\begin{proof}
	Define $\xi'$ by
	\[
	\xi' =  \xi - \eta - \frac{\eta}{2\langle \eta \rangle}
	\]
	so that
	$
	\nabla_{\eta} \Phi_{f-} = 2 \xi'.
	$
	If $|\xi'| \ll 1$, then there holds
	\begin{align*}
		|\Phi_{f-}(\xi,\eta)|
		= \left| \langle \eta \rangle + \frac{|\eta|^2}{\langle \eta \rangle} + 2\eta \cdot \xi' \right| 
		\ge \frac{1}{\langle \eta \rangle} \Big( 1 + 2 |\eta|^2 - 2 |\eta| \langle \eta \rangle |\xi'| \Big)
		\gtrsim 1.
	\end{align*}
	We define
	\begin{align*}
		\chi_{\cS_{f-}}(\xi,\eta) 
		&:= \varphi \Big(\delta^{-1}|\nabla_{\eta} \Phi_{f-}(\xi,\eta)|\Big), \\
		\chi_{\cT_{f-}}(\xi,\eta) &:= 1 - \chi_{\cS_{f-}}(\xi,\eta),
	\end{align*}
	for sufficiently small $\delta>0$.
	Then we get $|\Phi_{f-}| \gtrsim 1$ on the support of $\chi_{\cS_{f-}}$ and $|\nabla_{\eta} \Phi_{f-}| \gtrsim 1$ on the support of $\chi_{\cT_{f-}}$. 
	Meanwhile, a direct computation yields $| \partial_{\xi,\eta}^{\alpha} \Phi_{f-} | \lesssim 1$
	for $|\alpha| \ge 2$, which implies
	\[
	|\partial_{\xi,\eta}^{\alpha} \chi_{\cS_{f-}}| + |\partial_{\xi,\eta}^{\alpha} \chi_{\cT_{f-}}| \lesssim 1.
	\]
	We conclude that \eqref{est:cutoff-f-} holds.
\end{proof}

\begin{lemma}\label{lem:cutoff-f+}
	For $\rho \in (0,1]$, there exist smooth cutoff functions $\chi_{\cT_{f+}}^{\rho}$, $\chi_{\cS_{f+}}^{\rho}$, and $\chi_{\cR_{f+}}^{\rho}$, taking values in $[0,1]$, satisfying
	\begin{itemize}
		\item $\chi_{\cT_{f+}}^{\rho} + \chi_{\cS_{f+}}^{\rho} + \chi_{\cR_{f+}}^{\rho} =1$.
		\item $\chi_{\cR_{f+}}^{\rho}$ is supported in $B_{2\delta_0}(\cR)$ and is of the form
		\[
			\chi_{\cR_{f+}}^{\rho}(\xi,\eta) = \chi \left( \frac{|\xi-\eta| - (\lambda -1 )R}{\rho} \right) \chi \left( \frac{\xi-\lambda \eta}{\rho} \right)
		\]
		for some compactly supported function $\chi$, taking values in $[0,1]$.
		\item Outside of $B_{2\delta_0}(\cR)$, the cutoff functions $\chi_{\cT_{f+}}^{\rho}$ and $\chi_{\cS_{f+}}^{\rho}$ are independent of $\rho$.
		\item For all $\alpha$, the following estimates hold. If $|(\xi,\eta)| \le M$, then
		\begin{equation}\label{est:cutoff-f+-low}
		\left| \partial_{\xi,\eta}^{\alpha} \frac{\chi_{\cS_{f+}}^{\rho}}{\Phi_{f+}} \right| + \ \left| \partial_{\xi,\eta}^{\alpha} \frac{\chi_{\cT_{f+}}^{\rho} \partial_{\eta} \Phi_{f+}}{|\partial_{\eta} \Phi_{f+}|^2} \right| \lesssim \frac{1}{(\rho + d ((\xi,\eta) , \cR_{f+}))^{|\alpha|+1}},
		\end{equation}
		and if $|(\xi,\eta)| \ge M$, then 
		\begin{equation}\label{est:cutoff-f+-high}
		\left| \partial_{\xi,\eta}^{\alpha} \frac{\chi_{\cS_{f+}}^{\rho}}{\Phi_{f+}} \right| + \ \left| \partial_{\xi,\eta}^{\alpha} \frac{\chi_{\cT_{f+}}^{\rho} \partial_{\eta} \Phi_{f+}}{|\partial_{\eta} \Phi_{f+}|^2} \right| \lesssim |\xi,\eta|^{|\alpha|}.
		\end{equation}
	\end{itemize}
\end{lemma}

\begin{proof}
	Recall that $\cR_{f+} = \{ \xi = \lambda \eta, \ |\eta| = R \}$, where $R$ is the positive number satisfying $R^2 \sqrt{1+R^2}=1$ and $\lambda = 1 + \frac{1}{2\langle R \rangle}$.
	
	We begin with the case when $|(\xi,\eta)| \le M$ and $d((\xi,\eta), \cR_{f+}) \le 2 \delta_0$. The main issue here is the presence of large space-time resonance set $\mathcal{R}_{f+} = \{ \xi = \lambda \eta, \ |\eta| = R \}$.
	We first define
	\[
	\chi_{\cR_{f+}}^{\rho}(\xi,\eta) := \varphi\left( \frac{|\xi-\eta| - (\lambda-1)R }{\delta_0 \rho} \right) \varphi\left( \frac{\xi - \lambda \eta }{\delta_0 \rho} \right),
	\]
	which is supported in $B_{2\delta_0 \rho}(\cR)$.
	Next, we parametrize $\cS_{f+}$ as
	\[
	\cS_{f+} = \{ p(r,\omega), \quad (r,\omega) \in \mathbb{R}^+ \times \mathbb{S}^2 \}
	\]
	where $p(r,\omega) = (\lambda(r)r\omega, r\omega)$ and $\lambda(r) = 1 + \frac{1}{2\langle r \rangle}$. Define
	\[
	\chi_{\cS_{f+}}^{\rho}(\xi,\eta) = \left(1- \chi_{\cR_{f+}}^{\rho}(\xi,\eta)\right) \chi_{2} \left( C_0 \frac{d((\xi,\eta) , \cT_{f+}) - d((\xi,\eta) , \cS_{f+})}{d((\xi,\eta) , \cR_{f+})} \right)
	\]
	where $C_0 >0$ is a large constant and $\chi_{2}$ is a smooth nondecreasing function on $\mathbb{R}$ satisfying $\chi_{2}=0$ on $(-\infty, -1)$ and $\chi_{2} = 1$ on $(1,\infty)$. Finally we define
	\[
		\chi_{\cT_{f+}}^{\rho}(\xi,\eta)= 1- \chi_{\cR_{f+}}^{\rho}(\xi,\eta) - \chi_{\cS_{f+}}^{\rho}(\xi,\eta).
	\]
	We would like to prove	 \eqref{est:cutoff-f+-low}. We focus on $\frac{\chi_{\cS_{f+}}^{\rho}}{\Phi_{f+}}$ as the other term could be estimated in a similar way. We may assume $d((\xi,\eta), \cR_{f+}) \ge 2 \delta_0 \rho$, since otherwise $\chi_{\cS_{f+}}^{\rho} =0$. Noting that $|(\xi,\eta)| \le M$, we have $|\partial_{\xi,\eta}^{\alpha} \Phi| \lesssim 1$ for any $\alpha$. 
	Hence it suffices to prove
	\begin{equation}\label{claim}
		|\Phi(\xi,\eta)| \gtrsim d((\xi,\eta), \cR_{f+})
	\end{equation}
	on the support of $\chi_{\cS_{f+}}^{\rho}$.
	
	Since $|\partial_{\xi,\eta} \Phi_{f+}| \gtrsim 1$ on support of $\chi_{\cS_{f+}}^{\rho}$, we obtain from the mean value theorem that
	\[
		|\Phi(\xi,\eta)| \gtrsim d((\xi, \eta), \cT_{f+}).
	\]
	As $\mathcal{T}_{f+}$ and $\mathcal{S}_{f+}$ intersect transversally, we have
	\[
		d((\xi,\eta), \cT_{f+}) +  d((\xi,\eta), \cS_{f+}) \gtrsim d((\xi,\eta), \cR_{f+}).
	\]
	on the support of $\chi_{\cS_{f+}}^{\rho}$ provided that $\delta_0$ is taken sufficiently small.
	
	Next, we consider the case when $|(\xi,\eta)| \le M$ and $d((\xi,\eta), \cR) \ge 2\delta_0$. We have $\chi^{\rho}_{\cR_{f+}} =0$, and the cutoff functions $\chi_{\cT_{f+}}^{\rho}$ and $\chi_{\cS_{f+}}^{\rho}$ can be constructed independently of $\rho$ by repeating the same proof of Lemma \ref{lem:cutoff-f-}.
	
	Finally, we construct the cutoffs for $|(\xi,\eta)| \ge M$. In this case, we may ignore the parameter $\rho$ and it suffices to consider the case when $(\xi,\eta)$ is large.
	As $r\to \infty$, we have
	\[
	p(r,\omega) = \left( \left[r + \frac{1}{2} \right] \omega , r\omega  \right) + \mathcal{O}\left(\frac{1}{r^2} \right)
	\]
	and
	\[
	\Phi_{f+} (p(r, \omega)) = r^2 - \frac{1}{\langle r \rangle} \ge 1.
	\]
	In view of this, we define
	\[
	\chi_{\cS_{f+}}(\xi,\eta) = \varphi\left( \delta_M^{-1}\left[ |\xi-\eta| - \frac{1}{2} \right]\right) 
	\]
	for sufficiently small $\delta_M$, depending on $M$. It follows that $|\Phi_{f+}| \gtrsim 1$ on the support of $\chi_{\cS}$. We then define $\chi_{\cT_{f+}}(\xi,\eta) = 1- \chi_{\cS_{f+}}(\xi,\eta)$. Then $|\partial_{\eta} \Phi_{f+}| \gtrsim 1$ on the support of $\chi_{\cT_{f+}}$.
	
	A direct computation yields
	\begin{align*}
		\left| \partial_{\xi} \Phi_{f+}\right| \lesssim | \eta|, \qquad
		\left| \partial_{\eta} \Phi_{f+} \right|    \lesssim \langle \xi-\eta \rangle, \qquad
		| \partial_{\xi,\eta}^{\alpha} \Phi_{f+} | \lesssim 1
	\end{align*}
	for $|\alpha| \ge 2$. Therefore we obtain \eqref{est:cutoff-f+-high}.
\end{proof}

\begin{lemma}\label{lem:cutoff-g}
	For $\rho \in (0,1]$, there exist smooth cutoff functions $\chi_{\cT_{g}}^{\rho}$ and $\chi_{\cS_{g}}^{\rho}$, taking values in $[0,1]$, satisfying
	\begin{itemize}
		\item $\chi_{\cT_{g}}^{\rho} + \chi_{\cS_{g}}^{\rho} =1$.
		\item $\chi_{\cT_{g}}^{\rho}$ and $\chi_{\cS_{g}}^{\rho}$ are supported in $B_{M\rho^{-1}}(0)$.
		\item For all $\alpha$, there holds
		\begin{equation}\label{est:cutoff-g}
			\left| \partial_{\xi,\eta}^{\alpha} \frac{\chi_{\cS_g}^{\rho}}{\Phi_g} \right| + \left| \partial_{\xi,\eta}^{\alpha} \frac{\chi_{\cT_g}^{\rho} \partial_{\eta} \Phi_g }{|\partial_{\eta} \Phi_g|^2} \right|
			\lesssim \rho^{-|\alpha|}
		\end{equation}
	\end{itemize}
\end{lemma}

\begin{proof}
	Recall that $\cS_{g} = \{ \xi = 0 \}$ and $\cR_g = \emptyset$. Although the space-time resonance set is empty, we note that $d(\cT_{g}, \cS_{g}) =0$, where these two sets are close to each other in the case when $|\xi| \ll 1$ and $|\eta| \gg 1$.
	To address this issue, we will use the bounds on high norms in the high frequency regime $|(\xi,\eta)| \ge M s^{\delta_3}$.  Then we only construct the cutoffs in the region $|(\xi,\eta)| \le M s^{\delta_3}$. In view of this, we construct the cutoff functions $\chi_{\cT_g}^{\rho}$ and $\chi_{\cS_{g}}^{\rho}$ on $|(\xi,\eta)|\le M \rho^{-1}$.
	We define
	\begin{align*}
	\chi_{\cS_{g}}^{\rho}(\xi,\eta) &= \varphi \Big( \frac{8M}{\rho} |\xi| \Big), \\
	\chi_{\cT_{g}}(\xi,\eta) &= 1- \chi_{\cS_{g}}(\xi,\eta).
	\end{align*}
	Then $|\Phi_{g}| \gtrsim 1$ on the support of $\chi_{\cS_g}$, and $|\partial_{\eta} \Phi_{g}| \gtrsim 1$ on the support of $\chi_{\cT_{g}}$. Moreover, each $\xi$-derivative applied on $\chi_{\cS_{g}}^{\rho}$ contributes $\rho^{-1}$.
	It follows for any $\alpha$ that
	\[
	|\partial_{\xi,\eta}^{\alpha} \chi_{\cS}^{\rho}(\xi,\eta)| \lesssim \rho^{-|\alpha|}.
	\]
	Noting that
	\begin{align*}
		\left| \partial_{\xi} \Phi_{g}\right| \lesssim \langle \xi -\eta \rangle, \qquad
		\left| \partial_{\eta} \Phi_{g} \right|    \lesssim |\xi|, \qquad
		| \partial_{\xi,\eta}^{\alpha} \Phi_{g} | \lesssim 1
	\end{align*}
	for $|\alpha| \ge 2$, we obtain \eqref{est:cutoff-g}.
\end{proof}

\section{Nonlinear iterative scheme}\label{sec:nonlineariteration}
The local existence theory can be obtained using standard arguments. Global solutions are constructed by a bootstrap argument. Given $N \ge 4000$, fix  small $1 \gg \delta_1 \gg \delta_3 >0$, which will be chosen precisely in Section \ref{sec:parameters}. We define the bootstrap norm for $T \ge 0$ as
\[
\| (u, v) \|_{Z_T} := \| u \|_{X_T} + \sum_{\pm} \| V_{\pm} \|_{Y_T}
\]
where
\begin{multline*}
	\| u \|_{X_T} := \sup_{t \in [0,T]} 
	\Big\{ 
	\| u(t) \|_{H^{N}}
	+ \langle t \rangle^{1/2 + 3 \delta_1} \| u(t) \|_{W^{1, \threeplus}} 
	+ \langle t \rangle^{1-3\delta_1}\| \widetilde{Z}_{\cO} u(t) \|_{W^{1, \sixminus}} \\
	+ \langle t \rangle^{-1/2} \| x f(t) \|_{H^1}
	+ \frac{1}{\log(2+t)} \left\| |x|^{1/2} f(t) \right\|_{L^2}
	\Big\}
\end{multline*}
and
\begin{align*}
	\| V_{\pm} \|_{Y_T} := \sup_{t \in [0,T]} 
	\Big\{ 
	\| V_{\pm}(t) \|_{H^{N-1}}
	+ \langle t \rangle^{1-3\delta_1}\|  V_{\pm}(t) \|_{W^{1, \sixminus}}
	+ \langle t \rangle^{-1/2} \| x g_{\pm} (t) \|_{L^2} 
	\Big\}.
\end{align*}
The bootstrap norms for $u$ and $V_{\pm}$ are chosen differently because the equation for $u$ includes a space-time resonant phase $\Phi_{f+}$, whereas the equation for $V_{\pm}$ does not. We note that estimates for non-outcome frequencies (e.g., $\widetilde{Z}_{\cO}u$) are better than those for outcome frequencies (e.g., $Z_{\cO}u)$. Although the Klein-Gordon part corresponds to a non-resonant case, its nonlinearity is $|u|^2$ and the phase function for the Schr\"{o}dinger part has a nontrivial space-time resonant set $\mathcal{R}_{f+}$, which makes it difficult to prove an optimal decay even for $V_{\pm}(t)$. It might be an interesting problem to obtain an asymptotic formula for $u$ and $v$ for large times.

Recall the assumptions \eqref{assumption:initialdata} on the initial data. We assume that $\epsilon_0$ is sufficiently small and
\begin{equation}\label{eq:bootstrapassumption}
	\| (u,v)\|_{Z_{T}} \le \epsilon_1:= \epsilon_0^{5/6}.
\end{equation}
Then we aim to prove the improved estimate
\begin{equation}\label{eq:bootstrapprop}
	\|(u,v) \|_{Z_{T}} \le C_0 \epsilon_0 + C_1 \epsilon_1^2,
\end{equation}
where $C_0$ and $C_1$ are constants independent of $T$.
This establishes the global wellposedness theory. Indeed, let $T_{\ast}$ be the maximal time of existence and set
\[
T^{+} = \sup \{ t \in (0, T_{\ast}) : \| (u,v) \|_{Z_{t}} \le 2 C_0 \epsilon_0 \}.
\]
For sufficiently small $\epsilon_0$, we have $T^{+}=T_{\ast}$. Otherwise, we have $T^{+} < T_{\ast}$ so that
\[
2C_0 \epsilon = \| (u ,v)\|_{Z_{T}} \le C_0 \epsilon_0 + C_1 \epsilon^2 \le C_0 \epsilon_0 + C_1 \epsilon_0^{5/3} \le \frac{3}{2} C_0 \epsilon_0,
\]
which is a contradiction. We conclude that $T^{+} = T_{\ast} = \infty$, otherwise we can continue the solution further by the local wellposedness theory.

Linear scattering can then be established. Indeed, we note for $t_{1} \ge t_{2} \ge 1$ that
\begin{align*}
	\| f(t_1) - f(t_2) \|_{H^{N}} 
	&= \left\|  \int_{t_{2}}^{t_{1}} e^{-is\Delta} u(s) v(s) \, ds \right\|_{H^{N}} \\
	&\lesssim \| uv \|_{L^{\left(\frac{1}{2} + \frac{3}{2} \delta_1\right)^{-1}}_{t} W^{N, \left( \frac{5}{6} -\delta_1 \right)^{-1}}_{x}}\\
	&\lesssim \sum_{\pm} \Big\| \| u\|_{H^N} \| a_{\pm}(\nabla) V_{\pm} \|_{L^{\left( \frac{1}{3} - \delta_1 \right)^{-1}}} + \| u\|_{L^{\left( \frac{1}{3} - \delta_1 \right)^{-1}}} \| a_{\pm}(\nabla) V_{\pm} \|_{H^{N}} \Big\|_{L^{\left(\frac{1}{2} + \frac{3}{2} \delta_1\right)^{-1}}_{t}} \\
	&\lesssim \epsilon_0^2 \Big\|  t^{-\frac{1}{2} - 3\delta_1} \Big\|_{L^{\left(\frac{1}{2} + \frac{3}{2} \delta_1\right)^{-1}}_{t}}\\
	&\lesssim \epsilon_0^2 t_2^{-3\delta_1 /2}
\end{align*}
Hence $f(t)$ has the unique limit $f_{\infty}$ in $H^{N}$. Taking $t_{1} \to \infty$, we obtain for $t \ge 1$ that
\[
\| u(t) - u_{\infty}(t) \|_{H^{N}} =\| f(t) - f_{\infty} \|_{H^{N}} \lesssim \epsilon_0 t^{-3 \delta_1/2},
\]
where $u_{\infty}(t):= e^{it\Delta} f_{\infty}$.

Similarly one can show for $t_1 \ge t_2 \ge 1$ that
\begin{align*}
	\| g(t_1) - g(t_2) \|_{H^{N-1}} 
	&= \left\|  \int_{t_{2}}^{t_{1}} e^{is\langle \nabla \rangle} |u(s)|^2 \, ds \right\|_{H^{N-1}} \\
	&\lesssim \Big\| |u|^2  \Big\|_{L^{\left(\frac{1}{2} + \frac{3}{2} \delta_1\right)^{-1}}_{t} W^{N, \left( \frac{5}{6} -\delta_1 \right)^{-1}}_{x}}\\
	&\lesssim \sum_{\pm} \Big\| \| u\|_{H^N} \| u \|_{L^{\left( \frac{1}{3} - \delta_1 \right)^{-1}}}  \Big\|_{L^{\left(\frac{1}{2} + \frac{3}{2} \delta_1\right)^{-1}}_{t}} \\
	&\lesssim \epsilon_0^2 \Big\| t^{-\frac{1}{2} - 3\delta_1} \Big\|_{L^{\left(\frac{1}{2} + \frac{3}{2} \delta_1\right)^{-1}}_{t}}\\
	&\lesssim \epsilon_0^2 t_2^{-3\delta_1 /2}
\end{align*}
which implies
\[
\| g_{\pm}(t) - g_{\pm, \infty}(t) \|_{H^{N-1}} \lesssim \epsilon_0 t^{-3 \delta_1/2}
\]
for some $g_{\pm, \infty} \in H^{N-1}$. It follows that
\[
\| v(t) - v_{\infty}(t) \|_{H^{N}}\lesssim \epsilon_0 t^{-3\delta_1/2}
\]
where $v_{\infty}(t):= \sum_{\pm} a_{\pm}(\nabla) e^{\mp it \langle \nabla \rangle} g_{\pm, \infty} \in H^{N}$. Therefore, it suffices to prove \eqref{eq:bootstrapprop} under the assumption \eqref{eq:bootstrapassumption}.

\subsection{Implications of bootstrap assumptions}
Interpolating the estimates in the bootstrap assumption \eqref{eq:bootstrapassumption}, we obtain the following lemma on the time decay of $u(t)$ and $V_{\pm}(t)$.
\begin{lemma} Under the bootstrap assumption \eqref{eq:bootstrapassumption}, there hold for $0 \le t \le T$ that
	\begin{align*}
		\| \widetilde{Z}_{\cO} u(t) \|_{W^{1,p}} 
		&\le \epsilon \langle t \rangle^{-3 (1/2-1/p)}, \qquad 2 \le p \le \sixminus, \\
		\| u(t) \|_{W^{1,p}} 
		&\le \epsilon \langle t \rangle^{-3 (1/2-1/p)}, \qquad 2 \le p \le \threeplus, \\
		\| V_{\pm}(t) \|_{W^{1,p}} 
		&\le \epsilon \langle t \rangle^{-3 (1/2-1/p)}, \qquad 2 \le p \le \sixminus.
	\end{align*}
\end{lemma}
The following estimates are useful when we integrate by parts in time.
\begin{lemma}
	Under the bootstrap assumption \eqref{eq:bootstrapassumption}, we obtain for $0 \le t \le T$ that
	\begin{align*}
		\|e^{\mp it\langle \nabla \rangle} \partial_t g_{\pm} \|_{L^p}
		&\le \epsilon^2 \langle t \rangle^{-3 ( 1 - 1/p)},  \qquad 1 \le p \le \left(\frac{2}{3} - 2 \delta_1 \right)^{-1}, \\
		\|e^{it\Delta} \partial_t f \|_{L^p} 
		&\le \epsilon^2 \langle t \rangle^{-3 (1-1/p)},  \qquad 1 \le p \le 2.
	\end{align*}
\end{lemma}

\begin{proof}
	The proof follows from the observation that $e^{it\Delta}\partial_t f = -ivu$ and $e^{\mp it\langle \nabla \rangle} \partial_t g_{\pm} = |u|^2$.
\end{proof}

\subsection{Contribution of initial data and small time term}\label{sec:smalltime}

For $t \ge 1$, we can write
\begin{align*}
	\widehat{f}(t,\xi) = \widehat{f}(1,\xi) + \widehat{F}_{\pm}(t,\xi)
\end{align*}
where
\begin{align}
	\widehat{f}(1,\xi) &= \widehat{u}_{0}(\xi) + \sum_{\pm} \int_{0}^{1} \int e^{is \Phi_{f\pm}(\xi, \eta)} a_{\pm}(\eta) \widehat{f}(s,\xi-\eta) \widehat{g}_{\pm}(s, \eta) \, d\eta ds, \label{eq:def-f1} \\
	\widehat{F}_{\pm}(t, \xi) &= \int_{1}^{t} \int e^{is \Phi_{f\pm}(\xi, \eta)} a_{\pm}(\eta) \widehat{f}(s,\xi-\eta) \widehat{g}_{\pm}(s, \eta) \, d\eta ds. \label{eq:def-F}
\end{align}

Similarly we can write
\begin{equation*}
	\widehat{g}(t,\xi) = \widehat{g}(1,\xi) + \widehat{G}(t,\xi)
\end{equation*}
where
\begin{align}
	\widehat{g}(1,\xi) &= \widehat{V}_0 (\xi) + \int_{0}^{1} \int e^{is \Phi_g (\xi,\eta)} \widehat{f}(s,\xi-\eta) \overline{\widehat{f}(s,-\eta)} \, d\eta ds, \label{eq:def-g1} \\
	\widehat{G}(t, \xi) &= \int_{1}^{t} \int e^{is \Phi_{g}(\xi, \eta)}  \widehat{f}(s,\xi-\eta) \overline{\widehat{f}(s, -\eta)} \, d\eta ds. \label{eq:def-G}
\end{align}
Then $f(1)$ and $g(1)$ can be easily estimated using the Sobolev embedding and dispersive estimates.
\begin{lemma}
	Let $f(1)$ and $g(1)$ be given by \eqref{eq:def-f1} and \eqref{eq:def-g1}, respectively. Under the bootstrap assumption \eqref{eq:bootstrapassumption}, there hold
	\[
	\| e^{it\Delta} f(1) \|_{X_{T}} \le C_0 \epsilon_0 + C_1 \epsilon^2, \qquad \| e^{-it\langle \nabla \rangle} g(1) \|_{X_{T}} \le C_0 \epsilon_0 + C_1 \epsilon^2
	\]
	for some absolute constants $C_0$ and $C_1$.
\end{lemma}

\subsection{High frequency cutoffs}\label{sec:highfreq}
In this section, we introduce the high frequency cutoffs. When $|(\xi,\eta)| \ge M s^{\delta_3}$, we obtain that
\begin{align*}
	\| P_{\ge k} u (s)\|_{H^3} 
	&\lesssim 2^{-k(N-3)} \| u(s)\|_{H^{N}}
	\lesssim \epsilon s^{-2}, \\
	\| P_{\ge k} V_{\pm}(s)\|_{H^3}
	&\lesssim 2^{-k(N-4)} \| V_{\pm}(s) \|_{H^{N-1}} \lesssim \epsilon s^{-2},
\end{align*}
for $k \ge \delta_3 \log_2 s + C$, provided that $\delta_3 \ge \frac{2}{N-4}$. These will be sufficient to control all norms without spatial weights since the projection $P_{\ge \delta_3 \log_2 s + C}$ is applied to at least one of the bilinear terms, yielding a harmless term. Indeed, one obtains for $p < 6$ that
\[
	\| e^{it\Delta} F_{\pm}(t) \|_{L^p} + \| e^{-it\langle \nabla \rangle} G(t) \|_{L^p} \lesssim \int_{1}^{t} (t-s)^{-3 \left( \frac{1}{2} - \frac{1}{p} \right)} s^{-2} ds \lesssim t^{-3 \left(\frac{1}{2} - \frac{1}{p} \right)}.
\]
Thus, for the decay estimates in Sections \ref{sec:F-decay} and \ref{sec:G-decay}, it suffices to focus on the case $|(\xi,\eta)| \le M s^{\delta_3}$ without loss of generality. We will omit the cutoff $\varphi(\frac{(\xi,\eta)}{Ms^{\delta_3}})$ to simplify the notation.

\section{Estimates for the Schr\"{o}dinger part}\label{sec:F}
In this section, we focus on estimating $F_{\pm}(t)$. We recall that it is defined as
\begin{equation*}
	\widehat{F}_{\pm}(t, \xi) = \int_{1}^{t} \int e^{is \Phi_{f\pm}(\xi, \eta)} a_{\pm}(\eta) \widehat{f}(s,\xi-\eta) \widehat{g}_{\pm}(s, \eta) \, d\eta ds. 
\end{equation*}

\subsection{Energy estimates}
\begin{proposition}
	Let $F_{\pm}$ be defined by \eqref{eq:def-F}. Under the bootstrap assumption \eqref{eq:bootstrapassumption}, we have
	\[
		\| F_{\pm}(t) \|_{H^{N}} \lesssim \epsilon^2
	\]
	for $1 \le t \le T$.
\end{proposition}

\begin{proof}
Using Lemma \ref{lem:disp-Sch}, we get
\begin{align*}
	\bigg\| \int_{1}^{t} & e^{-is\Delta} u(s) a_{\pm}(\nabla) V_{\pm}(s) \,  ds \bigg\|_{H^{N}} \\
	&\lesssim \left\| u a_{\pm}(\nabla) V_{\pm} \right\|_{L^{\left(\frac{1}{2} + \frac{3}{2} \delta_1\right)^{-1}}_{s} W^{N, \left( \frac{5}{6} -\delta_1 \right)^{-1}}_{x}} \\
	&\lesssim \Big \| \| u \|_{H^{N}} \| V_{\pm} \|_{L_x^{\left( \frac{1}{3} - \delta_1 \right)^{-1}}} \Big\|_{L^{\left( \frac{1}{2} + \frac{3}{2} \delta_1 \right)^{-1}}_s}
		+ \Big \| \| u \|_{L_x^{\left( \frac{1}{3} - \delta_1 \right)^{-1}}} \| V_{\pm} \|_{H^{N-1}}  \Big\|_{L^{\left( \frac{1}{2} + \frac{3}{2} \delta_1 \right)^{-1}}_{s}} \\
	&\lesssim \epsilon^2 \left\|  s^{-\frac{1}{2} - 3 \delta_1 } \right\|_{L^{\left( \frac{1}{2} + \frac{3}{2} \delta_1 \right)^{-1}}_{s}} \\
	&\lesssim \epsilon^2,
\end{align*}
	which concludes the proof.
\end{proof}

\subsection{Localization estimates}
\begin{proposition}\label{prop:F-xH1}
	Let $F_{\pm}$ be defined by \eqref{eq:def-F}. Under the bootstrap assumption \eqref{eq:bootstrapassumption}, we have
	\[
		\| x F_{\pm}(t) \|_{H^1} \lesssim  \epsilon^2 t^{1/2}, \qquad \Big\| |x|^{1/2}F_{\pm}(t) \Big\|_{L^2} \lesssim \epsilon^2 \log (1+t),
	\]
	for $1 \le t \le T$.
\end{proposition}

\begin{proof}
We omit the subscript $\pm$ throughout the proof since the estimates are symmetric in both cases. Observe that for $2^{m-1} \le t_1 \le t_2 \le 2^m$, we have
\begin{align*}
	\| F(t_2) - F(t_1) \|_{L^2}
	&\lesssim \int_{t_1}^{t_2} \| u(s) \|_{L^{\threeplus}} \| V(s) \|_{L^{\sixminus}} \, ds \\
	&\lesssim \epsilon^2 \int_{t_1}^{t_2} s^{-3/2} ds \\
	&\lesssim \epsilon^2 2^{-m/2}.
\end{align*}
Next we show
\[
	\| xF(t_2) -xF(t_1) \|_{H^1}
	\lesssim \epsilon^2 2^{m/2}
\]
for $2^{m-1} \le t_1 \le t_2 \le 2^{m}$. We can write

\begin{align}
	\partial_{\xi} \widehat{F}(t_2) - \partial_{\xi} \widehat{F}(t_1)
	=&  \int_{t_1}^{t_2} \int e^{is \Phi_{f} } s \partial_{\xi} \Phi_{f}  \, a(\eta) \widehat{f}(s,\xi- \eta) \widehat{g}(s, \eta) \, d\eta ds \label{eq:hatF-1deriv-a}\\
	&+  \int_{t_1}^{t_2} \int e^{is \Phi_{f}} a(\eta) \partial_{\xi} \widehat{f}(s,\xi- \eta) \widehat{g}(s, \eta) \, d\eta ds. \label{eq:hatF-1deriv-b}
\end{align}
Applying Lemma \ref{lem:disp-Sch}, we bound
\begin{align*}
	\| \mathcal{F}_{\xi}^{-1} \eqref{eq:hatF-1deriv-b} \|_{H^1}
	&\lesssim \| (e^{is\Delta} x f ) a(\nabla) V \|_{L^2_{s}([t_1, t_2]) W^{1,6/5}_x} \\
	&\lesssim \Big\| \| xf \|_{H^1} \|V \|_{L^3} \Big\|_{L^2_{s}([t_1, t_2])} \\
	&\lesssim \epsilon^2  \| 1 \|_{L^2_{s}([t_1, t_2])} \\
	&\lesssim \epsilon^2 2^{m/2}.
\end{align*}
To bound \eqref{eq:hatF-1deriv-a}, we note that $\partial_{\xi} \Phi_{f}(\xi,\eta) = 2 \eta$. We get
\begin{align*}
	\left\| \mathcal{F}^{-1}_{\xi} \eqref{eq:hatF-1deriv-a} \right\|_{H^1} 
	&\lesssim \left\| \mathcal{F}^{-1}_{\xi} \int_{t_1}^{t_2}   \int  e^{is \Phi_{f} } s \partial_{\xi} \Phi_{f}  \, a(\eta)
	\widehat{f}(s,\xi- \eta) \widehat{g}(s, \eta) \, d\eta  ds  \right\|_{H^{1}} \\
	&\lesssim  \int_{t_1}^{t_2} s 
	\| u(s)  \|_{W^{1,\left(\frac{1}{3} - \delta_1 \right)^{-1}}}
	\| V(s) \|_{W^{1, \left(\frac{1}{6} + \delta_1 \right)^{-1}}} 
	\, ds \\
	&\lesssim \epsilon^2 \int_{t_1}^{t_2} s^{-1/2} ds\\
	&\lesssim \epsilon^2 2^{m/2}.
\end{align*}
It follows that
\[
\| x F(t)\|_{H^1}
\lesssim \sum_{1 \le m \le \log (1+t) }  \sup_{2^{m-1} \le t_{1} \le t_{2} \le 2^{m}} \| xF(t_2) - x F(t_1) \|_{L^2} \lesssim t^{1/2}.
\]
On the other hand, we obtain from the interpolation that
\begin{align*}
	\Big\| |x|^{1/2} F(t) \Big\|_{L^2} 
	&\lesssim \sum_{1 \le m \le \log (1+t) }  \sup_{2^{m-1} \le t_{1} \le t_{2} \le 2^{m}} \Big\| |x|^{1/2}F(t_2) - |x|^{1/2} F(t_1) \Big\|_{L^2} \\
	&\lesssim \sum_{1 \le m \le \log (1+t) } \| F(t_2) - F(t_1) \|_{L^2}^{1/2} \| x F(t_2) - x F(t_1) \|_{L^2}^{1/2} \\
	&\lesssim \log (1+t),
\end{align*}
which concludes the proof of Proposition \ref{prop:F-xH1}. 
\end{proof}

\subsection{Decay estimates}\label{sec:F-decay}
\begin{proposition}\label{prop:F-L6-}
	Let $F_{\pm}$ be defined by \eqref{eq:def-F}. Under the bootstrap assumption \eqref{eq:bootstrapassumption}, we have
	\[
		\| e^{it\Delta} \widetilde{Z}_{\cO} F_{+}(t) \|_{W^{1, \sixminus}} + \| e^{it\Delta} F_{-}(t) \|_{W^{1,\sixminus}} \lesssim \epsilon^2 t^{-1 + 3\delta_1}
	\]
	for $1 \le t \le T$.
\end{proposition}

\begin{proof}

Estimating both terms are symmetric, noting that we are in the nonresonant regime. We insert $\widetilde{\chi}_{\cO}(\xi)$ and drop the subscript $\pm$ for notational convenience. In view of Section \ref{sec:highfreq}, we may assume $|(\xi,\eta)| \le M s^{\delta_3}$ without loss of generality. We also recall the smooth cutoffs $1= \chi_{\cT} + \chi_{\cS}$, see Lemma \ref{lem:cutoff-f-} and \ref{lem:cutoff-f+} for the construction.
We write
\begin{align}
	\widetilde{\chi}_{\cO}(\xi)\widehat{F}(t,\xi)
	=& \int_{1}^{t} \int  \widetilde{\chi}_{\cO}(\xi) \chi_{\cS}(\xi,\eta)  e^{is \Phi_{f}} a(\eta) \widehat{f}(s,\xi-\eta) \widehat{g}(s,\eta) \, d\eta ds \label{eq:u-L6-low-awayT}\\
	& + \int_{1}^{t} \int  \widetilde{\chi}_{\cO}(\xi) \chi_{\cT}(\xi,\eta)   e^{is \Phi_{f}} a(\eta) \widehat{f}(s,\xi-\eta) \widehat{g}(s,\eta) \, d\eta ds. \label{eq:u-L6-low-awayS}
\end{align}
If we are away from $\cT$, we integrate by parts in time to get
\begin{align}
	\eqref{eq:u-L6-low-awayT}
	=&\int \widetilde{\chi}_{\cO}(\xi) \chi_{\cS}(\xi,\eta)  \frac{1}{i \Phi_f}  
	e^{it \Phi_f} a(\eta) \widehat{f}(t,\xi-\eta) \widehat{g}(t,\eta) \, d\eta
	\label{eq:u-L6-low-awayT-bdy} \\
	&- \int_{1}^{t} \int \widetilde{\chi}_{\cO}(\xi) \chi_{\cS}(\xi,\eta) \frac{1}{i \Phi_f}
	e^{is \Phi_f} a(\eta) \widehat{f}(s,\xi-\eta) \partial_s \widehat{g}(s,\eta) \, d\eta ds \label{eq:u-L6-low-awayT-cubic} \\
	&+ (\textrm{similar or easier terms}), \nonumber
\end{align}
where `similar terms' correspond to the case when $\partial_s$ hits $\widehat{f}$, and `easier terms' correspond to the boundary term at $s=1$, or those when $\partial_s$ hits one of the cutoff functions.
Let $$m_{1\pm} = \widetilde{\chi}_{\cO}(\xi)\chi_{\cS_{f\pm}}(\xi,\eta) \varphi( s^{-\delta_3} (\xi,\eta)) \frac{1}{i\Phi_{f\pm}}a_{\pm}(\eta).$$
Applying Bernstein inequalities \eqref{ineq:Bernstein} and Lemma \ref{lem:Coifman-Meyer}, the boundary term at $s=t$ is bounded by
\begin{align*}
	\| e^{it\Delta} \mathcal{F}^{-1}\eqref{eq:u-L6-low-awayT-bdy}\|_{W^{1, \sixminus}}
	&= \left\| T_{m_{1\pm}}(u(t), V(t)) \right\|_{W^{1, \sixminus}} \\
	&\lesssim t^{\delta_3} t^{\delta_3 \left( 1 - 3\delta_1 \right)} \| T_{m_{1\pm}} (u(t), V(t)) \|_{L^2} \\
	&\lesssim t^{\delta_3 \left( 2 - 3\delta_1 \right)}
	t^{A\delta_3} \| u(t) \|_{L^{\threeplus}} \| V(t) \|_{L^{\sixminus}} \\
	&\lesssim \epsilon^2 t^{\delta_3 \left(2  - 3\delta_1 \right)}
	t^{A\delta_3} t^{-\frac{3}{2}}  \\
	&\lesssim \epsilon^2 t^{-1 + 3\delta_1},
\end{align*}
provided that $(A+2)\delta_3 - 3\delta_3 \delta_1 - 3\delta_1 \le 1/2$.
Applying Lemma \ref{lem:disp-Sch}, Bernstein inequalities \eqref{ineq:Bernstein}, and Lemma \ref{lem:Coifman-Meyer}, we have 
\begin{align*}
	\|e^{it\Delta} &\mathcal{F}^{-1} \eqref{eq:u-L6-low-awayT-cubic} \|_{W^{1,\sixminus}} \\
	&\lesssim \left\| \int_{1}^{t} e^{i(t-s)\Delta} T_{m_{1\pm}} \Big(u , e^{- is\langle \nabla \rangle} \partial_s g  \Big)  ds \right\|_{W^{1, \sixminus}} \\
	&\lesssim \int_{1}^{t} (t-s)^{-1 + 3\delta_1} s^{\delta_3} \left\| T_{m_{1\pm}} \Big(u , e^{- is\langle \nabla \rangle} \partial_s g \Big) \right\|_{L^{\left( \frac{5}{6} - \delta_1 \right)^{-1}}} ds \\
	&\lesssim \int_{1}^{t} (t-s)^{-1 + 3\delta_1} s^{\delta_3} s^{\delta_3 /2} \left\| T_{m_{1\pm}} \Big(u , e^{- is\langle \nabla \rangle} \partial_s g  \Big) \right\|_{L^{(1-\delta_1)^{-1}}} ds \\
	&\lesssim \int_{1}^{t} (t-s)^{-1 + 3\delta_1} s^{\delta_3} s^{\delta_3 /2}
	s^{A \delta_3} \| u(s) \|_{L^{\threeplus}} \|e^{- is\langle \nabla \rangle} \partial_s g \|_{L^{\frac{3}{2}}} \, ds \\
	&\lesssim \epsilon^3  \int_{1}^{t} (t-s)^{-1 + 3\delta_1} s^{\delta_3} s^{\delta_3 /2}
	s^{A \delta_3} s^{-\frac{1}{2} - 3\delta_1} s^{-1} \, ds \\
	&\lesssim  \epsilon^3 t^{-1 + 3\delta_1},
\end{align*}
provided that $(A+3/2)\delta_3 - 3\delta_1 < 1/2$.

If we are away from $\cS$, we integrate by parts in $\eta$. This gives us
\begin{align}
	\eqref{eq:u-L6-low-awayS}
	&= -\int_{1}^{t} \int \widetilde{\chi}_{\cO}(\xi) \chi_{\cT}(\xi,\eta) \varphi\left(\frac{(\xi,\eta)}{Ms^{\delta_3}}\right) \frac{\partial_{\eta} \Phi}{is |\partial_{\eta} \Phi |^2} e^{is \Phi}
	a(\eta) \widehat{f}(s,\xi-\eta) \partial_{\eta} \widehat{g}(s,\eta)  \, d\eta ds \label{eq:u-L6-low-awayS-after} \\
	&\quad + (\textrm{symmetric or easier terms}). \nonumber
\end{align}
where `symmetric term' corresponds to the case when $\partial_{\eta}$ hits $\widehat{f}$ and `easier terms' corresponds to the case when $\partial_{\eta}$ hits one of the cutoff functions.
Define $$m_{2\pm} 
= \widetilde{\chi}_{\cO}(\xi) \chi_{\cT_{f\pm}} \varphi \left(\frac{(\xi,\eta)}{Ms^{\delta_3}}\right) \frac{\partial_{\eta} \Phi_{f\pm}}{i|\partial_{\eta} \Phi_{f\pm}|^2} a_{\pm}(\eta).$$
Then
\begin{align*}
	\| e^{it\Delta} &\mathcal{F}^{-1}\eqref{eq:u-L6-low-awayS-after} \|_{W^{1,\sixminus}} \\
	&\lesssim \left\| \int_{1}^{t} e^{i(t-s)\Delta} \frac{1}{s} T_{m_{2\pm}}\Big( u, e^{- is\langle \nabla \rangle} x g \Big) \, ds \right\|_{W^{1, \sixminus}} \\
	&\lesssim \int_{1}^{t} (t-s)^{-1 +3\delta_1} s^{-1} s^{\delta_3} \Big\|T_{m_{2\pm}}\Big( u, e^{- is\langle \nabla \rangle} x g \Big)  \Big \|_{L^{\left(\frac{5}{6}-\delta_1\right)^{-1}}} \, ds \\
	&\lesssim \int_{1}^{t} (t-s)^{-1 +3\delta_1} s^{-1} s^{\delta_3}
	s^{A\delta_3} \| u(s)\|_{L^{\threeplus}} \| x g(s)\|_{L^2} \, ds \\
	&\lesssim \epsilon^2 \int_{1}^{t} (t-s)^{-1 +3\delta_1} s^{-1} s^{\delta_3}
	s^{A\delta_3}  s^{-1/2 - 3\delta_1}  s^{1/2}\, ds \\
	&\lesssim \epsilon^2 t^{-1 + 3\delta_1},
\end{align*}
provided that $(A+1)\delta_3 - 3\delta_1 <0$.
This completes the proof of Proposition \ref{prop:F-L6-}.
\end{proof}

\begin{proposition}\label{prop:F-L3+}
	Let $F_{+}$ be defined by \eqref{eq:def-F}. Under the bootstrap assumption \eqref{eq:bootstrapassumption}, we have
	\[
	\| e^{it\Delta} F_{+}(t) \|_{W^{1, \threeplus}} \lesssim \epsilon^2 t^{-1/2 - 3\delta_1}
	\]
	for $1 \le t \le T$.
\end{proposition}

\begin{proof}
Note that $F_{-}$ was already estimated in the previous proposition. We estimate $F_{+}(t)$ here. For notational convenience, we drop the subscript $+$. First cutoff within a distance of order $\delta_0$ of $\mathcal{R}$ by writing $1= \chi_{\cR}^{1} + (1-\chi_{\cR}^{1})$. Write
\begin{align}
	F(t,\xi)
	=& \int_{1}^{t} \int \left( 1- \chi_{\cR}^{1} \right) e^{is \Phi_f} a(\eta) \widehat{f}(s,\xi-\eta) \widehat{g}(s,\eta) \, d\eta ds
	\label{eq:u-L3-awayR} \\
	& + \int_{1}^{t} \int  \chi_{\cR}^{1}  e^{is \Phi_f} a(\eta) \widehat{f}(s,\xi-\eta) \widehat{g}(s,\eta) \, d\eta ds
	\label{eq:u-L3-nearR}
\end{align}
The nonresonant term \eqref{eq:u-L3-awayR} can be estimated using the same argument as in the proof of Proposition \ref{prop:F-L6-}. We now focus on the resonant term \eqref{eq:u-L3-nearR}. Note that $|(\xi,\eta)| \lesssim 1$. Also, on $\operatorname{supp} \chi_{\cR}^{1}$, we have $\eta \not \in \cO$ and $\xi-\eta \not \in \cO$ since the resonances are separated. We further introduce time-dependent cutoffs $1= \chi_{\cR}^{s^{-\delta_2}} + \chi_{\cS}^{s^{-\delta_2}} + \chi_{\cT}^{s^{-\delta_2}}$ which were constructed in Lemma \ref{lem:cutoff-f+}. It follows that
\begin{align}
	\eqref{eq:u-L3-nearR} 
	=& \int_{1}^{t} \int \chi_{\cR}^{s^{-\delta_2}}  (\xi,\eta) e^{is \Phi_f} a(\eta) \widetilde{\chi}_{\cO} (\xi-\eta)\widehat{f}(s,\xi-\eta) \widehat{g}(s,\eta) \, d\eta ds \label{eq:u-L3-nearR-nearR}
	\\
	&+ \int_{1}^{t} \int  \chi_{\cR}^{1} \chi_{\cS}^{s^{-\delta_2}}(\xi,\eta) e^{is \Phi_f} a(\eta) \widetilde{\chi}_{\cO} (\xi-\eta)\widehat{f}(s,\xi-\eta) \widehat{g}(s,\eta) \, d\eta ds
	\label{eq:u-L3-nearR-awayT}
	\\
	&+ \int_{1}^{t} \int \chi_{\cR}^{1} \chi_{\cT}^{s^{-\delta_2}}(\xi,\eta) e^{is \Phi_f} a(\eta) \widetilde{\chi}_{\cO} (\xi-\eta)\widehat{f}(s,\xi-\eta) \widehat{g}(s,\eta) \, d\eta ds.
	\label{eq:u-L3-nearR-awayS}
\end{align}

Recalling the definition of $\chi_{\cR}^{s^{-\delta_2}}$ in Lemma \ref{lem:cutoff-f+}, the term near the space-time resonant set $\mathcal{R}$ can be bounded by
\begin{align*}
	\| e^{it\Delta} &\mathcal{F}^{-1}\eqref{eq:u-L3-nearR-nearR} \|_{W^{1,\threeplus}} \\
	&= \left\| \int_{1}^{t} e^{i(t-s)\Delta} T_{\chi(s^{\delta_2} (\xi-\lambda \eta))} \Big(e^{is \Delta} \widetilde{Z}_{\cO}\chi(s^{\delta_2}(|\nabla|-(\lambda-1)R)) f, a(\nabla) V \Big) \, ds \right\|_{W^{1, \threeplus}} \\
	&\lesssim \int_{1}^{t} (t-s)^{-\frac{1}{2} - 3\delta_1} \| \chi(s^{\delta_2} (|\nabla| -(\lambda-1)R) f(s) \|_{L^2} \| V (s) \|_{L^{\sixminus}} \, ds \\
	&\lesssim \epsilon^2  \int_{1}^{t} (t-s)^{-\frac{1}{2} - 3\delta_1} s^{-1 +3\delta_1} s^{-\delta_2/6} \log (1+s) \, ds \\
	&\lesssim \epsilon^2 t^{-\frac{1}{2} - 3\delta_1 },
\end{align*}
provided that $\delta_2 > 18 \delta_1$. Here we used
\[
	\| \chi(s^{\delta_2} (|\nabla|-R')) f(s) \|_{L^2} \lesssim s^{-\delta_2 / 6} \Big\| |x|^{1/2} f(s) \Big\|_{L^2} \lesssim \epsilon s^{-\delta_2 /6}\log (1+s)
\]
and
\[
	\| V(s)\|_{L^{\sixminus}} \lesssim \epsilon s^{-1 + 3\delta_1}.
\]
For the term \eqref{eq:u-L3-nearR-awayT} away from $\cT$, we integrate by parts in $s$ to obtain
\begin{align}
	\eqref{eq:u-L3-nearR-awayT}
	=&\int \chi_{\cR}^{1} \chi_{\cS}^{t^{-\delta_2}} e^{it \Phi_{f}} \frac{1}{i\Phi_{f}} (\xi,\eta) a(\eta) \widetilde{\chi}_{\cO} (\xi-\eta)\widehat{f}(t,\xi-\eta) \widehat{g}(t,\eta) \, d\eta ds 
	\label{eq:u-L3-nearR-awayT-bdy}
	\\
	&- \int_{1}^{t} \int \chi_{\cR}^{1} \chi_{\cS}^{s^{-\delta_2}} e^{is \Phi_{f}} \frac{1}{i\Phi_{f}} (\xi,\eta) a(\eta) \widetilde{\chi}_{\cO} (\xi-\eta)\widehat{f}(s,\xi-\eta)\partial_{s} \widehat{g}(s,\eta) \, d\eta ds 
	\label{eq:u-L3-nearR-awayT-cubic}
	\\
	&+(\textrm{similar or easier terms}). \nonumber
\end{align}
Let $$m_{3}^{t}(\xi,\eta) = \chi_{\cR}^{1} \chi_{\cS_{f+}}^{t^{-\delta_2}} \frac{1}{i\Phi_{f+}}(\xi,\eta).$$
The boundary term at time $t$ is estimated as
\begin{align*}
	\| e^{it\Delta} \mathcal{F}^{-1}\eqref{eq:u-L3-nearR-awayT-bdy}\|_{W^{1, \threeplus}}
	&= \| T_{m_{3}^t} (\widetilde{Z}_{\cO} u, V) \|_{W^{1, \threeplus}} \\
	&\lesssim \| T_{m_{3}^t} (\widetilde{Z}_{\cO} u, V) \|_{L^{\left( \frac{1}{3} + 2\delta_1\right)^{-1}}} \\
	&\lesssim t^{A\delta_2} \| \widetilde{Z}_{\cO} u \|_{L^{\sixminus}} \| V\|_{L^{\sixminus}} \\
	&\lesssim \epsilon^2 t^{-2 + 6 \delta_1 + A \delta_2} \\
	&\lesssim \epsilon^2  t^{-\frac{1}{2} - 3\delta_1},
\end{align*}
provided that $9\delta_1 + A\delta_2 < 3/2$.
The cubic term is bounded by
\begin{align*}
	\| e^{it\Delta} &\mathcal{F}^{-1} \eqref{eq:u-L3-nearR-awayT-cubic} \|_{W^{1,\threeplus}} \\
	&= \left\| \int_{1}^{t} e^{i(t-s)\Delta} T_{m_{3}^s} (\widetilde{Z}_{\cO} u, e^{- i s \langle \nabla \rangle} \partial_s V) \, ds  \right\|_{W^{1,\threeplus}} \\
	&\lesssim \int_{1}^{t} (t-s)^{-\frac{1}{2} - 3\delta_1 } \| T_{m_{3}^s} (\widetilde{Z}_{\cO} u, e^{- i s \langle \nabla \rangle} \partial_s V) \|_{W^{1, \left( \frac{2}{3} + \delta_1 \right)^{-1}}} \, ds \\
	&\lesssim \int_{1}^{t} (t-s)^{-\frac{1}{2} - 3\delta_1 } \| T_{m_{3}^s} (\widetilde{Z}_{\cO} u, e^{- i s \langle \nabla \rangle} \partial_s V) \|_{L^{\left( \frac{5}{6} + \delta_1 \right)^{-1}}} \, ds \\
	&\lesssim \int_{1}^{t} (t-s)^{-\frac{1}{2} - 3\delta_1 }
		s^{A \delta_2} \| \widetilde{Z}_{\cO} u(s) \|_{L^{\sixminus}} \| e^{-is\langle \nabla \rangle} \partial_s V(s) \|_{L^{3/2}}  \, ds \\
	&\lesssim \epsilon^3 \int_{1}^{t} (t-s)^{-\frac{1}{2} - 3\delta_1} s^{A \delta_2} s^{-1 + 3\delta_1 }  s^{-1} \, ds \\
	&\lesssim \epsilon^3 t^{-\frac{1}{2} - 3\delta_1},
\end{align*}
provided that $A\delta_2 + 3\delta_1 < 1$.

For the term away from $\cS$, we integrate by parts in $\eta$. It follows that
\begin{align}
	\eqref{eq:u-L3-nearR-awayS}
	=& - \int_{1}^{t} \int \chi_{\cR}^{1} \chi_{\cT}^{s^{-\delta_2}} \frac{\partial_{\eta} \Phi_{f}}{is |\partial_{\eta} \Phi_{f}|^2} e^{is\Phi_{f}} 
	\widetilde{\chi}_{\cO} (\xi-\eta)  \widehat{f}(\xi-\eta) \partial_{\eta} \widehat{g}(\eta) \, d\eta ds
	\label{eq:u-L3-nearR-awayS-after}
	\\
	 &+ (\textrm{symmetric or easier terms}). \nonumber
\end{align}
Then
\begin{align*}
	\| e^{it\Delta} &\mathcal{F}^{-1} \eqref{eq:u-L3-nearR-awayS-after}\|_{W^{1,\threeplus}} \\
	&= \left\| \int_{1}^{t} e^{i(t-s)\Delta} \frac{1}{s} T_{m_{4}^s} ( \widetilde{Z}_{\cO} u, e^{-is\langle \nabla \rangle} x g) \, ds \right\|_{W^{1, \threeplus}} \\
	&\lesssim \int_{1}^{t} (t-s)^{-\frac{1}{2}-3\delta_1} s^{-1} \| T_{m_{4}^s} ( \widetilde{Z}_{\cO} u, e^{-is\langle \nabla \rangle} x g)\|_{W^{1, \left(\frac{2}{3} + \delta_1 \right)^{-1}}} \, ds \\
	&\lesssim \int_{1}^{t} (t-s)^{-\frac{1}{2}-3\delta_1} s^{-1} s^{A\delta_2} \| \widetilde{Z}_{\cO} u(s) \|_{L^{\sixminus}} \| x g(s) \|_{L^2} \, ds  \\
	&\lesssim \epsilon^2 \int_{1}^{t} (t-s)^{-\frac{1}{2}-3\delta_1} s^{-1} s^{A\delta_2} s^{-1 +3\delta_1} s^{\frac{1}{2}}  \, ds \\
	&\lesssim \epsilon^2 t^{-\frac{1}{2} -3\delta_1},
\end{align*}
provided that $3\delta_1 + A \delta_2 < 1/2$, where $$m_{4}^s = \chi_{\cR}^{1} \chi_{\cT_{f+}}^{s^{-\delta_2}} \frac{\partial_{\eta} \Phi_{f+}}{i|\partial_{\eta} \Phi_{f+}|^2}.$$
This concludes the proof of Proposition \ref{prop:F-L3+}.
\end{proof}

\section{Estimates for the Klein-Gordon part}\label{sec:G}
In this section, we establish the bounds on $G$. We note that this corresponds to a nonresonant case. However, the decay and localization estimates for $u$ remain weak due to the presence of space-time resonances, which makes it difficult to prove optimal decay estimates. Recall that
\begin{equation*}
	\widehat{G}(t, \xi) = \int_{1}^{t} \int e^{is \Phi_{g}(\xi, \eta)}  \widehat{f}(s,\xi-\eta) \overline{\widehat{f}(s, -\eta)} \, d\eta ds.
\end{equation*}

\subsection{Energy estimates}
\begin{proposition}
	Let $G$ be defined by \eqref{eq:def-G}. Under the bootstrap assumption \eqref{eq:bootstrapassumption}, we have
	\[
	\|G(t)\|_{H^{N-1}} \lesssim \epsilon^2
	\]
	for $1 \le t \le T$.
\end{proposition}
\begin{proof}
Using Lemma \ref{lem:disp-KG}, we get
\begin{align*}
	\left\| \int_{1}^{t} e^{is\langle \nabla \rangle} |u(s)|^2 \, ds \right\|_{H^{N-1}} 
	&\lesssim \Big\| |u|^2 \Big\|_{L^{\left(\frac{1}{2} + \frac{3}{2} \delta_1\right)^{-1}}_{t} W^{N, \left( \frac{5}{6} -\delta_1 \right)^{-1}}_{x}} \\
	&\lesssim \Big \| \| u \|_{H^{N}} \| u\|_{L^{\left( \frac{1}{3} - \delta_1 \right)^{-1}}} \Big\|_{L^{\left( \frac{1}{2} + \frac{3}{2} \delta_1 \right)^{-1}}_t} \\
	&\lesssim \epsilon^2 \left\| \langle t \rangle^{-\frac{1}{2} - 3 \delta_1 } \right\|_{L^{\left( \frac{1}{2} + \frac{3}{2} \delta_1 \right)^{-1}}_{t}} \\
	&\lesssim \epsilon^2,	
\end{align*}
which concludes the proof.
\end{proof}

\subsection{Localization estimates}
\begin{proposition}\label{prop:G-xL2}
	Let $G$ be defined by \eqref{eq:def-G}. Under the bootstrap assumption \eqref{eq:bootstrapassumption}, there holds
	\[
	\| x G(t) \|_{L^2} \lesssim  \epsilon^2 t^{1/2}
	\]
	for $1 \le t \le T$.
\end{proposition}

\begin{proof}
Using Plancherel theorem, it is equivalent to estimate $\| \partial_{\xi} \widehat{G}(t)\|_{L^2}$. One can write

\begin{align}
	\partial_{\xi} \widehat{G}(t, \xi)
	=&  \int_{1}^{t} \int e^{is \Phi_{g} } s \partial_{\xi} \Phi_{g}  \, \widehat{f}(s,\xi- \eta) \overline{\widehat{f}(s,-\eta)} \, d\eta ds
	\label{eq:hatG-1deriv-a}
	\\
	&+ \int_{1}^{t} \int e^{is \Phi_{g}} \partial_{\xi} \widehat{f}(s,\xi- \eta) \overline{\widehat{f}(s,-\eta)} d\eta ds.
	\label{eq:hatG-1deriv-b}
\end{align}
Using Lemma \ref{lem:disp-KG}, we can show that
\begin{align*}
	\|\eqref{eq:hatG-1deriv-b}\|_{L^2}
	&\lesssim \left\| \int_{1}^{t} e^{is\langle \nabla \rangle} (e^{is\Delta} xf(s)) u(s) \, ds \right\|_{L^2_{s}([1,t])} \\
	&\lesssim \Big\| (e^{is\Delta} xf) u \Big\|_{L^2_{s}([1,t]) W^{5/6, 6/5}_x} \\
	&\lesssim \Big\| \| u\|_{W^{1,3}} \| xf \|_{H^1} \Big\|_{L^2_s ([1,t])} \\
	&\lesssim \epsilon^2 \| 1 \|_{L^2_{s}([1,t])} \\
	&\lesssim \epsilon^2 t^{1/2}.
\end{align*}
For \eqref{eq:hatG-1deriv-a}, we use $1= \chi_{\cO} + \widetilde{\chi}_{\cO}$ to split the frequency space. It follows that
\begin{align*}
	\eqref{eq:hatG-1deriv-a} = \int_{1}^{t} \int e^{is \Phi_{g} } s \partial_{\xi} \Phi_{g}  \Big(\chi_{\cO}(\xi-\eta) + \widetilde{\chi}_{\cO} (\xi-\eta)\Big) \widehat{f}(s,\xi- \eta) 
	\Big(\chi_{\cO}(\eta) + \widetilde{\chi}_{\cO} (\eta)\Big) \overline{\widehat{f}(s,-\eta)} \, d\eta ds.
\end{align*}
We may use better estimates for the nonoutcome frequencies. Noting that $$|\partial_{\xi} \Phi_g| = \left| \frac{\xi}{\langle \xi \rangle} - 2\xi +2 \eta \right| \lesssim \langle \xi -\eta \rangle,$$ we use Young's inequality to bound
\begin{align*}
	\bigg\| \int_{1}^{t} &s e^{is\langle \nabla \rangle} \left[ T_{\partial_{\xi} \Phi_g} \Big( \widetilde{Z}_{\cO} u(s), u(s) \Big)
	+ T_{\partial_{\xi} \Phi_g} \Big( Z_{\cO}u(s), \widetilde{Z}_{\cO} u(s) \Big) \right] ds \bigg\|_{L^2}\\
	&\lesssim \int_{1}^{t} s \| \widetilde{Z}_{\cO} u(s) \|_{W^{1, \left( \frac{1}{6} + \delta_1 \right)^{-1}}} \| u(s) \|_{W^{1, \left( \frac{1}{3} - \delta_1 \right)^{-1}}} \, ds \\
	&\lesssim \epsilon^2 \int_{1}^{t} s \, s^{-1 + 3\delta_1} s^{-\frac{1}{2} - 3\delta_1} \, ds \\
	&\lesssim \epsilon^2 t^{1/2} .
\end{align*}
It remains to consider the term corresponding to $\chi_{\cO}(\xi-\eta) \chi_{\cO}(\eta)$. Since $\xi-\eta$ and $\eta$ are not far away from $\cR$, we are in the low frequency regime $|(\xi,\eta)| \lesssim 1$. In this case, $\operatorname{dist}(\cT_g, \cS_g)\gtrsim 1$, so we may use time-independent cutoffs $1= \chi_{\cS_g} + \chi_{\cT_g}$ by taking $\rho=1$ in Lemma \ref{lem:cutoff-g}.  We write
\begin{align}
	\int_{1}^{t} & \int   e^{is \Phi_g} s \partial_{\xi} \Phi_g  \chi_{\cO} (\xi-\eta) \widehat{f}(s,\xi-\eta) \chi_{\cO}(\eta) \overline{\widehat{f}(s,-\eta)} \, d\eta ds \nonumber \\
	&= \int_{1}^{t} \int  \chi_{\cS_{g}} (\xi,\eta)  e^{is \Phi_g} s \partial_{\xi} \Phi_g  \chi_{\cO} (\xi-\eta) \widehat{f}(s,\xi-\eta) \chi_{\cO}(\eta) \overline{\widehat{f}(s,-\eta)} \, d\eta ds 
	\label{eq:hatG-1deriv-a-OO-awayT}
	\\
	&\quad + \int_{1}^{t} \int   \chi_{\cT_{g}} (\xi,\eta)  e^{is \Phi_g} s \partial_{\xi} \Phi_g  \chi_{\cO} (\xi-\eta) \widehat{f}(s,\xi-\eta) \chi_{\cO}(\eta) \overline{\widehat{f}(s,-\eta)} \, d\eta ds.
	\label{eq:hatG-1deriv-a-OO-awayS}
\end{align}
If we are away from $\cT$, we integrate by parts in time to get
\begin{align}
	\eqref{eq:hatG-1deriv-a-OO-awayT}
	=&\int \chi_{\cS_{g}}(\xi,\eta)  \frac{1}{i \Phi_g} e^{it\Phi_g} t \partial_{\xi} \Phi_g \chi_{\cO} (\xi-\eta) \widehat{f}(t,\xi-\eta) \chi_{\cO}(\eta) \overline{\widehat{f}(t,-\eta)} \, d\eta ds. 
	\label{eq:hatG-1deriv-a-OO-awayT-bdy}
	\\
	&- \int_{1}^{t} \int \chi_{\cS_{g}}(\xi,\eta)  \frac{1}{i \Phi_g} e^{is\Phi_g} s \partial_{\xi} \Phi_g \chi_{\cO} (\xi-\eta) \partial_s\widehat{f}(s,\xi-\eta) \chi_{\cO}(\eta) \overline{\widehat{f}(s,-\eta)} \, d\eta ds.
	\label{eq:hatG-1deriv-a-OO-awayT-cubic}
	\\
	&+ (\textrm{symmetric or easier terms}). \nonumber
\end{align}
Let $$m_{5} = \frac{\chi_{\cS_{g}} }{i\Phi_g} \partial_{\xi} \Phi_g \chi_{\cO}(\xi-\eta) \chi_{\cO}(\eta).$$ The boundary term at $s=t$ is estimated as follows. Applying Bernstein inequalities \eqref{ineq:Bernstein} and Lemma \ref{lem:Coifman-Meyer}, we have
\begin{align*}
	\| \eqref{eq:hatG-1deriv-a-OO-awayT-bdy}\|_{L^2}
	&= t \left\| T_{m_{5}}(u(t), u(t)) \right\|_{L^2} \\
	&\lesssim t \| T_{m_{5}} (u(t), u(t)) \|_{L^{3/2}} \\
	&\lesssim t\| u\|_{L^3}^2 \\
	&\lesssim \epsilon^2.
\end{align*}
For the cubic term, applying Lemma \ref{lem:disp-KG}, Bernstein inequalities \eqref{ineq:Bernstein}, and Lemma \ref{lem:Coifman-Meyer}, we have 
\begin{align*}
	\|\eqref{eq:hatG-1deriv-a-OO-awayT-cubic}\|_{L^2}
	&\lesssim \left\| \int_{1}^{t} s e^{is\langle\nabla\rangle} T_{m_{5}} \Big(e^{is\Delta} \partial_s f , u \Big)  ds \right\|_{L^2} \\
	&\lesssim \int_{1}^{t} s \left\| T_{m_{5}} \Big( e^{is\Delta} \partial_s f , u \Big)\right\|_{L^{6/5}} ds \\
	&\lesssim \int_{1}^{t} s \| u(s) \|_{L^3} \| e^{is\Delta} \partial_s f \|_{L^2} \, ds \\
	&\lesssim \epsilon^3 \log t .
\end{align*}
If we are away from $\cS_{g}$, we integrate by parts in $\eta$. This gives us
\begin{align}
	\eqref{eq:hatG-1deriv-a-OO-awayS}
	=& \int_{1}^{t} \int  \chi_{\cT_{g}}(\xi,\eta)   \frac{\partial_{\eta} \Phi_g}{is |\partial_{\eta} \Phi_g |^2}e^{is \Phi_g} s \partial_{\xi} \Phi_g  \chi_{\cO} (\xi-\eta) \partial_{\eta}\widehat{f}(s,\xi-\eta) \chi_{\cO}(\eta) \overline{\widehat{f}(s,-\eta)} \, d\eta ds 
	\label{eq:hatG-1deriv-a-OO-awayS-after}
	\\
	&+ (\textrm{symmetric or easier terms}) \nonumber .
\end{align}
Defining $$m_{6} =  \frac{\chi_{\cT_{g}} \partial_{\eta} \Phi_g}{i|\partial_{\eta} \Phi_g|^2}  \partial_{\xi}\Phi_g \chi_{\cO}(\xi-\eta) \chi_{\cO}(\eta),$$ we get
\begin{align*}
	\| \eqref{eq:hatG-1deriv-a-OO-awayS-after} \|_{L^2} 
	\lesssim \Big\| \|xf\|_{H^1} \| u\|_{W^{1,3}} \Big\|_{L^2}
	\lesssim \epsilon^2 t^{1/2},
\end{align*}
noting that $\langle \partial_{\eta} \rangle m_{6}(\xi,\eta) \lesssim 1$.
This completes the proof of Proposition \ref{prop:G-xL2}.
\end{proof}

\subsection{Decay estimates}\label{sec:G-decay}
\begin{proposition}\label{prop:G-L6-}
	Let $G$ be defined by \eqref{eq:def-G}. Under the bootstrap assumption \eqref{eq:bootstrapassumption}, there holds
	\[
	\| e^{-it\langle \nabla \rangle} G(t) \|_{W^{1, \sixminus}} \lesssim \epsilon^2 t^{-1 + 3\delta_1}
	\]
	for $1\le t \le T$.
\end{proposition}

\begin{proof}
In view of Section \ref{sec:highfreq}, we may assume $|(\xi,\eta)| \le M s^{\delta_3}$.
Recall from Lemma \ref{lem:cutoff-g} that we constructed the cutoffs $\chi_{\cT_{g}}^{\rho}$ and $\chi_{\cS_{g}}^{\rho}$.
It remains to estimate
\begin{align}
	&\int_{1}^{t} \int  \chi_{\cS_{g}}^{s^{-\delta_3}} e^{is \Phi_g} \widehat{f}(s,\xi-\eta) \overline{\widehat{f}(s,-\eta)} \, d\eta ds
		\label{eq:V-L6-low-awayT}
	\\
	&+ \int_{1}^{t} \int  \chi_{\cT_{g}}^{s^{-\delta_3}} e^{is \Phi_g} \widehat{f}(s,\xi-\eta) \overline{\widehat{f}(s,-\eta)} \, d\eta ds.
		\label{eq:V-L6-low-awayS}
\end{align}
If we are away from $\cT_g$, we integrate by parts in time to get
\begin{align}
	\eqref{eq:V-L6-low-awayT}
	=&\int \chi_{\cS_{g}}^{t^{-\delta_3}}   \frac{1}{i \Phi_g} e^{it\Phi_g} \widehat{f}(t,\xi-\eta) \overline{\widehat{f}(t,-\eta)} \, d\eta
		\label{eq:V-L6-low-awayT-bdy}
	 \\
	&- \int_{1}^{t} \int \chi_{\cS_{g}}^{s^{-\delta_3}} \frac{1}{i \Phi_g} e^{is\Phi_g} \partial_s \widehat{f}(s,\xi-\eta) \overline{\widehat{f}(s,-\eta)} \, d\eta ds 
		\label{eq:V-L6-low-awayT-cubic}
	\\
	&+ (\textrm{symmetric or easier terms}).
\end{align}
Let $$m_{7} = \varphi \left(\frac{(\xi,\eta)}{Ms^{\delta_3}}\right) \chi_{\cS_{g}}^{s^{-\delta_3}}  \frac{1}{i\Phi_g}.$$ 
We apply Bernstein inequalities \eqref{ineq:Bernstein} and Lemma \ref{lem:Coifman-Meyer} to bound
\begin{align*}
	\| e^{-it\langle\nabla\rangle} &\mathcal{F}^{-1}\eqref{eq:V-L6-low-awayT-bdy}\|_{W^{1, \sixminus}} \\
	&= \left\| T_{m_{7}}(u(t), u(t)) \right\|_{W^{1, \left( \frac{1}{6} + \delta_1 \right)^{-1}}} \\
	&\lesssim  t^{\delta_3} t^{\delta_3 \left( \frac{1}{2} - 3\delta_1 \right)} \| T_{m_{7}} (u(t), u(t)) \|_{L^{\left( \frac{2}{3} - 2\delta_1 \right)^{-1}}} \\
	&\lesssim t^{\delta_3 \left( \frac{3}{2} - 3\delta_1 \right)}
		t^{A\delta_3} \| u (t) \|_{L^{\left( \frac{1}{3} - \delta_1 \right)^{-1}}}^{2} \\
	&\lesssim  \epsilon^2 t^{\delta_3 \left( \frac{3}{2} - 3\delta_1 \right)}
	t^{A\delta_3} t^{-1-6\delta_1} \\
	&\lesssim \epsilon^2 t^{-1 + 3\delta_1},
\end{align*}
provided that $(A + 3/2) \delta_3 - 3\delta_1 \delta_3 - 9 \delta_1 \le 0$.

Applying Lemma \ref{lem:disp-KG} and \ref{lem:Coifman-Meyer}, we have 
\begin{align*}
	\| e^{-it\langle \nabla \rangle} 
	&\mathcal{F}^{-1} \eqref{eq:V-L6-low-awayT-cubic} \|_{W^{1, \sixminus}} \\
	&\lesssim \left\| \int_{1}^{t} e^{-i(t-s)\langle\nabla\rangle} T_{m_{7}} \Big(e^{is\Delta} \partial_s f , u \Big)  ds \right\|_{W^{1, \left( \frac{1}{6} + \delta_1  \right)^{-1}}} \\
	&\lesssim \int_{1}^{t} (t-s)^{-1 + 3\delta_1} s^{3\delta_3} \left\| T_{m_{7}} \Big( e^{is\Delta} \partial_s f , u \Big)\right\|_{L^{\left( \frac{5}{6} - \delta_1 \right)^{-1}}} ds \\
	&\lesssim \int_{1}^{t} (t-s)^{-1 + 3\delta_1} s^{3\delta_3}
		s^{A \delta_3} \| e^{is\Delta} \partial_s f(s)  \|_{L^{2}} \| u(s) \|_{L^{\threeplus}} ds \\
	&\lesssim \epsilon^3 \int_{1}^{t} (t-s)^{-1 + 3\delta_1} s^{3\delta_3}
	s^{A \delta_3} s^{-3/2} s^{-1/2 -3\delta_1} \, ds \\
	&\lesssim \epsilon^3 t^{-1 + 3\delta_1},
\end{align*}
provided that $(A+3) \delta_3 - 3\delta_1 <1$.

If we are away from $\cS_{g}$, we integrate by parts in $\eta$. This gives us
\begin{align}
	\eqref{eq:V-L6-low-awayS}
	=& \int_{1}^{t} \int  \chi_{\cT_{g}}^{s^{-\delta_3}} \frac{\partial_{\eta} \Phi_g}{is |\partial_{\eta} \Phi_g |^2} e^{is \Phi_g} \partial_{\eta} \widehat{f}(s,\xi-\eta)	\overline{\widehat{f}(s,-\eta)} \, d\eta ds
	\label{eq:V-L6-low-awayS-after}
	 \\
	&+ (\textrm{symmetric or easier terms}).
\end{align}
Define $$m_{8} = \chi_{\cT_{g}}^{s^{-\delta_3}} \varphi \left( \frac{(\xi,\eta)}{Ms^{\delta_3}}\right) \frac{\partial_{\eta} \Phi_g}{i|\partial_{\eta} \Phi_g|^2}.$$
Then
\begin{align*}
	\| e^{-it\langle\nabla\rangle} &\mathcal{F}^{-1}\eqref{eq:V-L6-low-awayS-after} \|_{W^{1, \sixminus}} \\
	&\lesssim \left\| \int_{1}^{t} e^{-i(t-s)\langle\nabla\rangle} \frac{1}{s} T_{m_{8}}\Big( e^{is\Delta} xf, u \Big) \, ds \right\|_{W^{1, \left( \frac{1}{6} + \delta_1 \right)^{-1}}} \\
	&\lesssim \int_{1}^{t} (t-s)^{-1 +3\delta_1} s^{-1} s^{3\delta_3} \| T_{m_{8}} (e^{is\Delta} xf, u) \, ds \|_{L^{\left(\frac{5}{6}-\delta_1\right)^{-1}}} \\
	&\lesssim \int_{1}^{t} (t-s)^{-1 +3\delta_1} s^{-1} s^{3\delta_3}
	s^{A\delta_3} \| xf(s)\|_{L^2} \| u(s)\|_{L^{\left(\frac{1}{3} -\delta_1 \right)^{-1}}} \, ds \\
	&\lesssim \epsilon^2 \int_{1}^{t} (t-s)^{-1 +3\delta_1} s^{-1} s^{3\delta_3}
	s^{A\delta_3} s^{1/2} s^{-1/2 - 3\delta_1} \, ds \\
	&\lesssim \epsilon^2 t^{-1 + 3\delta_1},
\end{align*}
provided that $(A+3)\delta_3 - 3 \delta_1 <0$.
This completes the proof of Proposition \ref{prop:G-L6-}.
\end{proof}

\section{Proof of Theorem \ref{thm:generalizedKGS}}\label{sec:generalizedKGS}

Finally, in this section, we sketch the proof of Theorem \ref{thm:generalizedKGS}. We first observe that $\nu(|k|)$ is of Klein-Gordon type so that similar dispersive estimates hold. Indeed,
$$\|e^{it \nu(|\nabla|)} f \|_{L^p} \lesssim t^{-3 \left( \frac{1}{2} - \frac{1}{p} \right)} \| f \|_{W^{ 5\left(\frac{1}{2} - \frac{1}{p} \right) + \epsilon, p'}}$$
and
$$\left\| \int_{0}^{t} e^{is\nu(|\nabla|)} F(s) \, ds \right\|_{L^2_{x}} \lesssim \| F \|_{L^{p'}_t W^{- \frac{1}{q} + \frac{1}{p} + \frac{1}{2} + \epsilon, q'}_x}$$
hold for $\epsilon >0$, $2 \le p \le \infty$, and $ 2/p + 3/q = 3/2$, see \cite[Lemma A.3]{DAncona2008} and \cite[Proposition 2.6]{Nguyen2024b} for proofs.

Next, we verify that the resonant sets have a similar structure. The phase functions are
\begin{equation*}
	\begin{split}
		\Phi_{f\pm} (\xi, \eta) = |\xi|^2 \mp \nu(|\eta|) - |\xi-\eta|^2 , \qquad
		\Phi_{g} (\xi, \eta) = \nu(|\xi|) - |\xi - \eta|^2 + |\eta|^2. 
	\end{split}
\end{equation*}
By direct computation, we obtain that
\begin{equation}\label{eq:resonance-f+-gen}
	\begin{split}
		\mathcal{T}_{f+} 
		&= \left\{ 2\xi \cdot \eta =\nu(|\eta|) + |\eta|^2 \right\},\\[2pt]
		\mathcal{S}_{f+} 
		&= \left\{ \xi = \eta \left( 1+ \frac{\nu'(|\eta|)}{2|\eta|} \right) \right\}, \\[2pt]
		\mathcal{R}_{f+} 
		&= \{ \xi = \lambda \eta, \ |\eta| = R \},
	\end{split}
\end{equation}
where $R$ is the unique positive number satisfying $R^2 - \nu(R) + \nu'(R) R = 0$, and $\lambda =1 + \frac{\nu'(R)}{2R}$. 
For $\Phi_{f-}$,  we have 
\begin{equation}\label{eq:resonance-f--gen}
	\begin{split}
		\mathcal{T}_{f-} 
		&= \left\{ 2\xi \cdot \eta =-\nu(|\eta|) + |\eta|^2 \right\},\\[2pt]
		\mathcal{S}_{f-} 
		&= \left\{ \xi = \eta \left( 1- \frac{\nu'(|\eta|)}{2|\eta|} \right) \right\}, \\[2pt]
		\mathcal{R}_{f-} 
		&= \emptyset,
	\end{split}
\end{equation}
and for $\Phi_{g}$, we get
\begin{equation}\label{eq:resonance-g-gen}
	\begin{split}
		\mathcal{T}_{g} 
		&= \left\{ 2\xi \cdot \eta = -\nu(|\xi|) +|\xi|^2  \right\}, \\[2pt]
		\mathcal{S}_{g} 
		&= \{ \xi = 0 \}, \\[2pt]
		\mathcal{R}_{g} 
		&= \emptyset.
	\end{split}
\end{equation}
Note that $R$ and $\lambda$ are determined by the zero of
\[
	I_{f\pm}(r) : = r^2 \mp \nu(r) \pm \nu'(r)r .
\]
Namely, if $\Phi_{f\pm}(\xi,\eta) =0 $ and $\partial_{\eta} \Phi_{f\pm}(\xi,\eta) =0$, then $I_{f+}(|\eta|) =0$.
Observe that $I_{f+}(r)$ has a unique zero $R \in (0,\infty)$, since $I_{f+}(0) = -\nu(0) <0$, $I_{f+}(\infty) =\infty$, and $I_{f+}'(r) = 2r + \nu''(r) r >0$. This yields \eqref{eq:resonance-f+-gen}.
Next, we have $I_{f-}(0) = \nu(0) > 0$, $I_{f-}(\infty) = \infty$, and $I_{f-}'(r) = 2r - \nu''(r) r  \ge 0$, so we get $I_{f-}(r) >0$ for all $r \ge 0$. This implies \eqref{eq:resonance-f--gen}. It is straightforward to check \eqref{eq:resonance-g-gen}.

Next, the set of outcome and germ frequencies are
\[
	\cO = B_{\lambda R}(0), \qquad \mathcal{G} = B_{R}(0) \cup B_{(\lambda -1)R}(0).
\]
The resonances are separated since $R \neq 0$ and $\lambda >1$.

Finally, the derivatives $\partial_{\xi, \eta}^{|\alpha|} \Phi_{\iota}$ may not be bounded due to the term from the Schr\"{o}dinger part. Then it can be checked that the bounds of $\partial_{\xi,\eta}^{|\alpha|} \Phi_{\iota}$ are similar to the previous case \eqref{eq:KGS}. Namely, for $2 \le |\alpha| \le N_0$, we have
\[
	|\partial_{\xi} \Phi_{f\pm} | \lesssim |\eta|, \qquad |\partial_{\eta} \Phi_{f\pm} | \lesssim \langle \xi-\eta \rangle, \qquad |\partial_{\xi,\eta}^{\alpha} \Phi_{f\pm}| \lesssim 1,
\]
and
\[
	|\partial_{\xi} \Phi_{g} | \lesssim \langle \xi-\eta \rangle, \qquad |\partial_{\eta} \Phi_{g} | \lesssim |\xi|, \qquad |\partial_{\xi,\eta}^{\alpha} \Phi_{g}| \lesssim 1.
\]
Hence all the cutoff functions can be constructed in the same way, see Lemma \ref{lem:cutoff-f-}, \ref{lem:cutoff-f+}, and \ref{lem:cutoff-g}. The rest of the proof follows by repeating the arguments in Sections \ref{sec:smalltime}, \ref{sec:highfreq}, \ref{sec:F}, and \ref{sec:G}.

\subsection*{Acknowledgement}
The author would like to thank Toan T. Nguyen for many helpful discussions.
The research is supported in part by the NSF under grant DMS-2349981. 
\bibliographystyle{plain2}

\end{document}